\documentclass[a4paper,11pt]{amsart}
\usepackage{amsmath,amsthm,amsfonts,amssymb,upgreek,multicol,mathrsfs,pifont,amscd,latexsym,cite,bm,geometry,color,fancyhdr,abstract,framed,appendix,hyperref,cleveref}
\usepackage{graphicx,algorithm,algorithmic,tabularx,url,float,booktabs,array,threeparttable,tikz,enumerate}
\usepackage{lineno}
\usepackage{xcolor}
\usetikzlibrary{backgrounds}
\usetikzlibrary{arrows}
\usetikzlibrary{shapes,shapes.geometric,shapes.misc,cd}

\usepackage[all]{xy}
\allowdisplaybreaks[4]
\newcommand{\mn}{\bm{n}}
\newcommand{\mbt}{\bm{t}}
\newcommand{\mv}{\mathcal{V}}
\newcommand{\me}{\mathcal{E}}
\newcommand{\mf}{\mathcal{F}}
\newcommand{\ms}{\mathbb{S}}
\newcommand{\mt}{\mathbb{T}}
\newcommand{\mr}{\mathbb{R}}

\newcommand{\bsi}{\bm{\sigma}}
\newcommand{\bta}{\bm{\tau}}
\newcommand{\mI}{\bm{I}}

\theoremstyle{theorem}
\newtheorem{theorem}{Theorem}[section]
\newtheorem{prop}{Proposition}[section]

\newtheorem{lemma}{Lemma}[section]
\theoremstyle{remark}

\newtheorem{remark}{Remark}[section]

\numberwithin{equation}{section}
\numberwithin{table}{section}
\numberwithin{figure}{section}

\DeclareMathOperator{\dev}{dev}
\DeclareMathOperator{\hess}{hess}
\DeclareMathOperator{\sym}{sym}
\DeclareMathOperator{\curl}{curl}
\DeclareMathOperator{\divergence}{div}
\renewcommand{\div}{\divergence}

\DeclareMathOperator{\grad}{grad}
\DeclareMathOperator{\tr}{tr}
\DeclareMathOperator{\Tr}{Tr}
\DeclareMathOperator{\vskw}{vskw}
\DeclareMathOperator{\mskw}{mskw}
\DeclareMathOperator{\rot}{rot}
\DeclareMathOperator{\skw}{skw}

\newcommand{\ST}{{\mathbb S \cap \mathbb T}}
\newcommand{\B}{{\mathbb B}}

\title[Discretizing Conformal Hessian Complexes]{Discretizing linearized Einstein-Bianchi system by symmetric and traceless tensors}
\author{Yuyang Guo}
\address{School of Mathematical Sciences, Peking University, Beijing 100871, P.R. China.}
\email{yuyangguo@stu.pku.edu.cn}

\author{Jun Hu}
\address{LMAM and School of Mathematical Sciences, Peking University, Beijing 100871, P.R. China.}
\email{hujun@math.pku.edu.cn}

\author{Ting Lin}
\address{School of Mathematical Sciences, Peking University, Beijing 100871, P.R. China.}
\email{lintingsms@pku.edu.cn}

\thanks{Jun Hu was supported by NSFC project No.12288101. Ting Lin was supported by NSFC project No.123B2014.}

\begin{document}
\maketitle 
\begin{abstract}
The Einstein-Bianchi system uses symmetric and traceless tensors to reformulate Einstein's original field equations. However, preserving these algebraic constraints simultaneously remains a challenge for numerical methods. This paper proposes a new formulation that treats the linearized Einstein-Bianchi system (near the trivial Minkowski metric) as the Hodge wave equation associated with the conformal Hessian complex. To discretize this equation, a conforming finite element conformal Hessian complex that preserves symmetry and traceless-ness simultaneously is constructed on general three-dimensional tetrahedral grids, and its exactness is proven.
\end{abstract}

\section{Introduction}
Einstein's field equations are of fundamental importance in general relativity, as they describe the curvature of spacetime induced by mass and energy. Therefore, the system has received long-term attention from the fields of mathematics and physics over the past decades. However, their highly nonlinear nature poses significant challenges in obtaining either exact or numerical solutions. Among various formulations, the \emph{Einstein-Bianchi system} is based on the decomposition of the Weyl tensor and the Bianchi identity, which constitutes a class of hyperbolic partial differential equations.

The linearized Einstein-Bianchi system \cite{quenneville2015new} reads as follows:
\begin{equation}\label{eq:EB}
    \begin{aligned}
        \dot{\bm E}+\curl\bm B=0,\ \div\bm E=0,\\
        \dot{\bm B}-\curl\bm E=0,\ \div\bm B=0.
    \end{aligned}
\end{equation}
Here $\bm E$ and $\bm B$ are both symmetric and traceless matrix-valued functions in three dimensions. 

Despite the formal resemblance between this system and Maxwell's equations, the key distinction is that the variables $\bm E$ and $\bm B$ in \eqref{eq:EB} are not vector-valued functions but symmetric and traceless matrix-valued functions. 
As a consequence, the primary challenge of \eqref{eq:EB} lies in constructing discrete matrix-valued functions that adhere to algebraic constraints — namely, being symmetric and traceless simultaneously — while preserving the differential property of being divergence-free as time evolves, which are also fundamental properties of Einstein's equations. Therefore, the methods for solving Maxwell's equations cannot be directly applied to the linearized Einstein-Bianchi system. 

This paper addresses the challenge by providing a discretization that imposes both algebraic constraints strongly. By introducing an auxiliary variable $\sigma(t)=\int_{0}^t\div\div\bm E ds$, we obtain the following system:
\begin{equation}\label{eq:HodgeEB2}
\begin{aligned}
    \begin{cases}
        \dot{\sigma}=\div\div\bm{E},\\
        \dot{\bm{E}}=-\dev\hess\sigma-\sym\curl\bm{B},\\
        \dot{\bm{B}}=\sym\curl\bm{E},
    \end{cases}
\end{aligned}
\end{equation}
with the scalar-valued function $\sigma$ and symmetric and traceless matrix-valued functions $\bm E$ and $\bm B$. 
Here $\dev\hess\sigma$ is the traceless part of the Hessian of $\sigma$. Note that the $\curl$ of a symmetric matrix-valued function is traceless. Therefore, $\sym\curl\bm B$ is both symmetric and traceless, and the system is well-defined. This system seeks $\sigma \in C^1([0,T], L^2(\Omega))$, $\bm{E} \in C^0([0,T], H(\div\div, \Omega; \ST))$, and $\bm{B} \in C^0([0,T], H(\sym\curl, \Omega; \ST))$. Here $H(\div\div,\Omega;\ST)$ consists of $L^2$ matrix-valued functions $\bm \sigma$ with $\div\div \bm \sigma $ in $L^2$, and $H(\sym\curl, \Omega; \ST)$ consists of $L^2$ matrix-valued functions $\bta$ with $\sym\curl \bta$ in $L^2$. Note that symmetry and traceless-ness are strongly imposed in these spaces. 

The following theorem, to be verified in \Cref{subsec:cplx}, guarantees that the newly proposed system \eqref{eq:HodgeEB2} is equivalent to the linearized Einstein-Bianchi system \eqref{eq:EB}.
\begin{theorem}\label{thm:Hodge=EB}
    Given initial conditions $(\sigma_0,\bm{E}_0,\bm{B}_0)$ and appropriate boundary conditions, the system \eqref{eq:HodgeEB2} is well-posed. Suppose that $\sigma_0=0$ and $\bm E_0$ and $\bm B_0$ are divergence-free, then, for all time, $\sigma=0$ and $\bm E$ and $\bm B$ remain divergence-free. Consequently, \eqref{eq:HodgeEB2} is equivalent to the linearized Einstein-Bianchi system \eqref{eq:EB}.    
\end{theorem}

Moreover, the system \eqref{eq:HodgeEB2} is the Hodge wave equation of the following conformal Hessian complex (see \cite[(52)]{Arnold2021} for the smooth version):
\begin{equation}\label{eq:complex3Dc}
\begin{aligned}
    P_{1}^{+}\xrightarrow[]{\subset}H^2(\Omega)& \xrightarrow[]{ \dev\hess}H(\sym\curl,\Omega;\ST)\\ &~~\xrightarrow[]{\sym\curl}H(\div\div,\Omega;\ST) \xrightarrow[]{ \div\div}L^2(\Omega) \rightarrow 0,
\end{aligned}
\end{equation}
where $P_{1}^{+}:=P_1\oplus{\rm {span}}\{\bm x^T\bm x\}$ denotes the kernel of the operator $\dev\hess$. 

The conformal Hessian complex, together with the Hodge wave equation \eqref{eq:HodgeEB2}, imposes both algebraic constraints of the linearized Einstein-Bianchi system. Therefore, for the linearized Einstein-Bianchi system, the conformal Hessian complex serves as the counterpart to the de Rham complex for Maxwell's equations. This indicates the importance of the discretization of the conformal Hessian complex. 

The key advantage of numerical methods based on the Hodge wave equation \eqref{eq:HodgeEB2} for the linearized Einstein-Bianchi system is that they preserve the algebraic constraints of symmetry and traceless-ness simultaneously. In contrast, the existing discretization methods for the linearized Einstein-Bianchi system are generally based on the Hodge wave equation that relaxes the algebraic constraints by merely enforcing symmetry of $\bm E$ or traceless-ness of $\bm B$. 
The associated complexes are the Hessian complex or the $\div\div$ complex, respectively \cite{Pauly2020,quenneville2015new}; see \cite{Hu2021a, hu2024family, Hu2022d, Chen2022a} for finite element discretizations, \cite{arf2021structure} for spline-based numerical methods, and \cite{hu2025distributional, Chen2025} for distribution-based discretizations. Still, none of them preserve both algebraic constraints.
 
The stable discretization of the Hodge wave equation \eqref{eq:HodgeEB2} concentrates on the construction of the finite element conformal Hessian complex. The main difficulty lies in the lack of a geometric basis of the five-dimensional space $\ST$ on tetrahedral or cuboid meshes (see \cite{MR4841479}), making characterizations of the key bubble function spaces elusive. In contrast, either $\mathbb{S}$ or $\mathbb{T}$ has a geometrically natural basis with respect to a tetrahedron \cite{Hu2015,Hu2021a}. 

\subsection*{Contributions}
One primary contribution of this paper is to reformulate the linearized Einstein-Bianchi system \eqref{eq:EB} as the Hodge wave equation associated with the conformal Hessian complex \eqref{eq:complex3Dc}. Such equivalence implies that the conformal Hessian complex serves as the analogue of the de Rham complex for Maxwell's equations. Another contribution is to discretize the conformal Hessian complex with conforming finite elements while preserving the cohomology. More precisely, we construct a family of finite element spaces $\Lambda_{k+1,h}\subset H(\sym\curl,\Omega;\ST)$ of polynomial degree $\le k+1$ and a family of finite element spaces $\Sigma_{k,h}\subset H(\div\div,\Omega;\ST)$ of polynomial degree $\le k$, such that the finite element conformal Hessian complex
\begin{align}\label{dct_cplx}
    P_{1}^{+} \xrightarrow[]{\subset}  U_{k+3,h} \xrightarrow[]{\dev\hess} \Lambda_{k+1,h} \xrightarrow[]{\sym\curl} \Sigma_{k,h} \xrightarrow[]{\div\div} P_{k-2}(\mathcal{T}_h)\xrightarrow[]{} 0
\end{align}
is exact together with a family of $H^2$ conforming finite element spaces \cite{Zhang2009} of polynomial degree $\le k+2$ and discontinuous piecewise polynomial spaces $P_{k-2}(\mathcal{T}_h)$ on tetrahedral grids $\mathcal{T}_h$ for $k\ge 6$. Here
$P_{1}^{+}:=P_1\oplus{\rm {span}}\{\bm x^T\bm x\}$ is the kernel of the operator $\dev\hess$. The error estimates of the scheme for the Hodge wave equation \eqref{eq:HodgeEB2}, based on the discrete conformal Hessian complex, are also analyzed. 

To address the challenge caused by the lack of a geometric basis of $\ST$, we give implicit characterizations of the $H(\div\div;\ST)$ bubble function space instead. We prove that the space spanned by the skew-symmetric part of the functions in the $H(\sym\curl;\mt)$ bubble function space is isomorphic to the $H(\div;\mr^3)$ bubble function space in \Cref{lem:vskw_symcurl}. Therefore, the $H(\div\div;\ST)$ bubble function space is the complement space of the $\sym\curl$ space of the $H(\sym\curl;\mt)$ bubble function space with respect to the $H(\div\div;\ms)$ bubble function space. Such a procedure can be regarded as a realization of the Belstein--Gelfand--Gelfand framework in \cite{Arnold2021}\cite{MR4700410}. As for the $H(\sym\curl;\ST)$ bubble function space, we manage to construct its explicit characterization in \Cref{lem:symcurl_basis}.

Besides, we derive the traces in two dimensions of the matrix-valued finite element spaces to ensure the conformity conditions, and we construct the trace complexes to analyze the regularity of the finite element functions on faces and edges. As a result, we obtain a thorough understanding of the $H(\sym\curl,\Omega;\ST)$ traces and the $H(\div\div,\Omega;\ST)$ traces, which are the generalization of the cases in \cite{Hu2022d}. 
The conformal strain complex in two dimensions is derived as one of the trace complexes, and its discretized version is used for the construction of the finite element conformal Hessian complex.

\enlargethispage{-4\baselineskip}
\subsection*{Notations}
Let $\mathcal{T}_h$ be a shape regular triangulation of $\Omega\subset\mr^3$ into tetrahedra. Let $\mv$ denote the set of all vertices, $\me$ the set of all edges and $\mf$ the set of all faces. Given $e\in\me$, let $\mbt$ denote the unit tangential vector along $e$, and let $\mn_1$ and $\mn_2$ denote two independent unit normal vectors such that $\mn_1\times \mn_2=\mbt$. Given $F\in\mf$, let $\mn$ denote the unit normal vector of $F$, and let $\mbt_1$ and $\mbt_2$ denote two independent unit tangential vectors on $F$ such that $\mbt_1\times\mbt_2=\mn$. Moreover, let $\mn_{\partial F}$ denote the outer normal vector of $\partial F$ on $F$ and $\mbt_{\partial F}$ denote the unit tangential vector of $\partial F$ such that $\mn_{\partial F}\times \mbt_{\partial F}=\mn$.  

Given $K\in\mathcal{T}_h$, let $\mf(K)$ be the set of all faces of $K$, $\me(K)$ be the set of all edges of $K$, $\mv(K)$ be the set of all vertices of $K$. Let $F_i$ denote the faces of $K$ and $\bm{x}_i$ denote the vertex opposite to $F_i$ with $0\leq i\leq 3$. Let $\lambda_i$ with $0\leq i\leq 3$ be the barycentric coordinates of $K$, and let $b_K=\lambda_0\lambda_1\lambda_2\lambda_3$. 
For the face $F$ opposite to the vertex $\bm x_i$, let $b_{F}=\lambda_j\lambda_l\lambda_m$, with $0\leq i\leq 3$ and $\{i,j,l,m\}$ being a permutation of $\{0,1,2,3\}$. 

Given a two-dimensional triangle $F\in\mf$, let $\me(F)$ be the set of all edges of $F$, $\mv(F)$ be the set of all vertices of $F$. Given $e \in \me(F)$, let $\mn=\mn_{\partial F}$ denote the unit normal vector of $e$ and $\mbt=\mbt_{\partial F}$ denote the unit tangential vector of $e$ in the two-dimensional plane. 

Throughout this paper, we denote the space of all symmetric $3\times 3$ matrices by $\ms$ and all traceless $3\times 3$ matrices by $\mt$. Given a domain $\Omega$, $L^2(\Omega; \mathbb{X})$ and $H^m(\Omega;\mathbb{X})$ represent the space $L^2$ and the Sobolev space respectively with values in $\mathbb{X}$. $P_k(\Omega ;\mathbb{X})$ denotes the set of polynomials of total degree not greater than $k$, with $\mathbb{X}$ being one of the following spaces: $\mr,\mr^3,\mr^{3\times 3}, \ms, \mt$ and $\ST$. If $\mathbb{X}=\mr$, then $L^2(\Omega)$ abbreviates $L^2(\Omega;\mathbb{X})$, similarly for $H^m(\Omega)$ and $P_k(\Omega)$.

The organization of this paper is as follows. \Cref{sec:cplx} introduces the conformal Hessian complex together with its traces. \Cref{sec:bubble} constructs the exact finite element conformal Hessian bubble complex. \Cref{sec:trace} introduces the finite element spaces corresponding to the trace complexes of the conformal Hessian complex. \Cref{sec:FEM3D} constructs a family of $H(\sym\curl;\ST)$ conforming finite element spaces and a family of $H(\div\div;\ST)$ conforming finite element spaces. Moreover, the constructed finite element conformal Hessian complex \eqref{dct_cplx} is shown to be exact provided that the domain is contractible. 
\Cref{sec:application} uses the proposed finite element spaces to discretize the Hodge wave equation \eqref{eq:HodgeEB2} and shows the error estimates. 

\section{The Conformal Hessian complex in three dimensions}\label{sec:cplx}
This section studies the properties of the conformal Hessian complex in three dimensions. In particular, we focus on the connection between its associated Hodge wave equation and the linearized Einstein-Bianchi system. We also derive the conformity conditions of each function space in the conformal Hessian complex and establish the trace complexes in two dimensions. The conformity conditions and the trace complexes will be used in the later construction of finite element complexes. To this end, we first define some operators in tensor calculus. We then introduce the conformal Hessian complex together with its associated Hodge wave equation, and conduct the analysis of its conformity conditions and trace complexes.

\subsection{Algebraic and differential operators}
This subsection reviews some basic notations including algebraic operators and differential operators for vectors and tensors in three dimensions. 

Let $\bm{I}$ denote the $n\times n$ identity matrix. Given a matrix $\bm{A}\in\mr^{n\times n}$, denote by
\begin{align*}
\sym\bm{A}=\frac{1}{2}(\bm{A}+\bm{A}^T),\;\skw\bm{A}=\frac{1}{2}(\bm{A}-\bm{A}^T),\;\dev \bm{A}=\bm{A}-\frac{1}{n} \Tr(\bm{A})\bm{I}
\end{align*}
the symmetric part, skew-symmetric part, traceless part of $\bm A$, respectively. 
For a vector $\bm{v}=(v_1,v_2,v_3)^T$, we define a skew-symmetric matrix as follows
\begin{align*}
\mskw \bm{v} = \begin{pmatrix}
0 &-v_3& v_2\\
v_3 &0& -v_1\\
-v_2&v_1&0 
\end{pmatrix}.
\end{align*}
Clearly, $\mskw$ is a bijection from the vector space to the skew-symmetric matrix space. 
For a matrix $\bm{A}\in\mr^{3\times 3}$, define 
\begin{align*}
    \vskw\bm{A}=\mskw^{-1}\skw\bm{A}.
\end{align*}

Next, we introduce some differential operators. 
For a matrix-valued function $\bm{A}$,
the $\curl$ and $\div$ operators apply row-wise to produce a matrix-valued function $\curl \bm{A}$ and a vector-valued function $\div \bm{A}$, respectively. 
For a vector-valued function $\bm{v} = (v_1,v_2,v_3)^T$ (treated as a column vector), the $\grad$ operator applies row-wise to produce a matrix-valued function 
\begin{align*}
\grad \bm{v}=\begin{pmatrix}
\partial_x v_1 &\partial_y v_1 &\partial_z v_1\\
\partial_x v_2 &\partial_y v_2 &\partial_z v_2\\
\partial_x v_3 &\partial_y v_3&\partial_z v_3 
\end{pmatrix}.
\end{align*}

For a plane $F$ in $\mr^3$ with the unit normal vector $\mn$ and the tangential vectors $\mbt_1,\mbt_2$, satisfying $\mbt_1\times\mbt_2=\mn$, suppose that the coordinate on $F$ is spanned by $\mbt_1,\mbt_2$. Define the projection operator $\Pi_F:\mr^3\xrightarrow{}\mr^2$ for a vector-valued function $\bm v$ as
\begin{equation*}
    \Pi_F \bm v=(\bm v\cdot\mbt_1,\bm v\cdot\mbt_2)^T,
\end{equation*}
and define the extension operator $E_F:\mr^2\xrightarrow{}\mr^3$ for a vector-valued function $\bm u=(u_1,u_2)^T$ as
\begin{equation*}
    E_F\bm u=u_1\mbt_1+u_2\mbt_2.
\end{equation*}
This allows us to define some differential operators on $F$. Define the surface gradient and the surface curl operator for a scalar-valued function $q$ as
\begin{align*}
\grad_F q  =\Pi_F\grad q,  \quad \curl_F q  =\Pi_F(\mn\times\grad q).
\end{align*}
Then \begin{equation*}
    (\grad_F q, \curl_F q)=0.
\end{equation*}

Define the surface rot operator for a vector-valued function $\bm{v}\in\mr^3$ as
\begin{align*}
\rot_F \bm{v} = (\mn\times\nabla)\cdot \bm{v}=\mn\cdot(\curl \bm{v}),
\end{align*}
and the surface divergence $\div_F$ as
\begin{align*}
\div_F \bm{v}=(\mn\times \nabla)\cdot(\mn\times \bm{v})=\rot_F(\mn\times \bm{v}).
\end{align*}
The surface rot and divergence operators for a vector-valued function $\bm u\in\mr^2$ are defined as
\begin{align*}
    \rot_F\bm u=\rot_F E_F\bm u,\; \div_F\bm u=\div_F E_F\bm u.
\end{align*}
The surface gradient $\grad_F$ applies row-wise on $\bm u\in\mr^2$ to produce a matrix-valued function $\grad_F \bm{u}$. 

In particular, suppose that $F$ is chosen as the $x-y$ plane with $\mn = (0,0,1)^T$. Then, these operators $\grad_F, \rot_F,$ $\curl_F,\div_F,\sym\grad_F$ are standard differential operators in two dimensions. Namely, for a scalar-valued function $q$ and a vector-valued function $\bm{u}=(u_1,u_2)^T$, the surface differential operators read as  $\grad_F q=(\partial_x q,\partial_y q)^T,\curl_F q=(-\partial_y q,\partial_x q)^T$, $\div_F\bm{u}=\partial_x u_1+\partial_y u_2$, $\rot_F \bm{u}=\partial_x u_2-\partial_y u_1$ and $\sym\grad_F \bm{u}=\frac{1}{2}(\grad_F\bm{u}+(\grad_F\bm{u})^T)$.

Given a matrix-valued function $\bm{A}$, the cross product of a vector from the left acts row-wise and from the right acts column-wise. The operator $\Pi_F$ from the left acts row-wise and from the right acts column-wise. The operators $\rot_F$ and $\div_F$ act row-wise.

We also present some important identities. Given a vector-valued function $\bm{v}$ and a matrix-valued function $\bm{A}$, the vector products commute with differentiation as follows:
\begin{align*}
 (\grad \bm{v})^T\mn=\grad(\bm{v}\cdot\mn),\\
 \grad \bm{v}\times \mn=\grad(\bm{v}\times \mn),\\
 (\curl \bm{A})^T\mn=\curl(\bm{A}^T\mn).
\end{align*}

\subsection{The Conformal Hessian complex}\label{subsec:cplx}
Recall that the conformal Hessian complex \eqref{eq:complex3Dc} in three dimensions from \cite{Arnold2021} is given as
\begin{align*}
    P_{1}^{+}\xrightarrow[]{\subset}H^2(\Omega) \xrightarrow[]{ \dev\hess}H(\sym\curl,\Omega;\ST) \xrightarrow[]{\sym\curl}H(\div\div,\Omega;\ST) \xrightarrow[]{ \div\div}L^2(\Omega) \rightarrow 0,
\end{align*}
where 
\begin{align*}
    H(\sym\curl,\Omega;\ST):&=\{\bta\in L^2(\Omega;\ST) \colon \sym\curl\bta\in L^2(\Omega;\ST)\},\\
    H(\div\div,\Omega;\ST):&=\{\bsi\in L^2(\Omega;\ST) \colon \div\div\bsi \in L^2(\Omega)\}.
\end{align*}

The $\div\div$ complex in three dimensions from \cite{Pauly2020} reads as:
\begin{equation}\label{eq:complex3Db}
    \begin{aligned}
    RT\xrightarrow[]{ \subset}H^1(\Omega;\mr^3) &\xrightarrow[]{ \dev\grad}H(\sym\curl,\Omega;\mt) \\&~~\xrightarrow[]{\sym\curl}H(\div\div,\Omega;\ms) \xrightarrow[]{\div\div}L^2(\Omega) \rightarrow 0.
\end{aligned}
\end{equation}
Here $
    RT:=\{\bm a+b\bm{x} \colon  \bm a\in\mr^3,b\in\mr\}
$ is the lowest order Raviart--Thomas shape function space in three dimensions, and 
\begin{align*}
    H(\sym\curl,\Omega;\mt):&=\{\bta\in L^2(\Omega;\mt) \colon \sym\curl\bta\in L^2(\Omega;\ms)\},\\
    H(\div\div,\Omega;\ms):&=\{\bsi\in L^2(\Omega;\ms) \colon \div\div\bsi \in L^2(\Omega)\}.
\end{align*}
Notice that the matrix-valued function spaces in both complexes are only different in their algebraic constraints, where the matrix-valued function spaces in the $\div\div$ complex are either symmetric or traceless, while those in the conformal Hessian complex are both symmetric and traceless. The conformal Hessian complex \eqref{eq:complex3Dc} is exact on a bounded Lipschitz contractible domain $\Omega$, given the exactness of the de Rham complex and the $\div\div$ complex \eqref{eq:complex3Db}, see \cite{Arnold2021}.

The Hodge wave equation \eqref{eq:HodgeEB2} associated with the conformal Hessian complex is critical for the construction of stable finite elements for the linearized Einstein-Bianchi system \eqref{eq:EB}. To illustrate this point, we first briefly review the existing approaches to \eqref{eq:EB}.

The formal resemblance between the linearized Einstein-Bianchi \eqref{eq:EB} system and Maxwell's equations motivates the reformulation of \eqref{eq:EB} as a Hodge wave equation as in \cite[Chapter 8.6]{MR3908678}. This enables the design of stable numerical discretizations. Specifically, the first such reformulation is based on the vector-valued de Rham complex \cite{quenneville2015new}:
\begin{equation}
\begin{aligned}
    P_1(\mr^2) \xrightarrow{\subset}H^1(\Omega;\mr^3)&\xrightarrow{\grad}H(\curl,\Omega;\mr^{3\times 3})\\&~~\xrightarrow{\curl}H(\div,\Omega;\mr^{3\times 3})\xrightarrow{}L^2(\Omega;\mr^3)\xrightarrow{}0,
\end{aligned}
\end{equation}
of which the Hodge wave equation reads:
\begin{equation}\label{eq:HodgeEB0}
\begin{aligned}
    \begin{cases}
        \dot{\sigma}=\div\bm{E},\\
        \dot{\bm{E}}=-\grad\sigma-\curl\bm{B},\\
        \dot{\bm{B}}=\curl\bm{E},
    \end{cases}
\end{aligned}
\end{equation}
with the scalar-valued function $\sigma$ and matrix-valued functions $\bm E$ and $\bm B$. 

Although this formulation is equivalent to the linearized Einstein-Bianchi system at the continuous level, it inherently fails to impose the algebraic constraints, since $\bm E$ and $\bm B$ are $\mr^{3\times 3}$ matrix-valued functions without built-in symmetric or traceless constraints. 

Such inability to intrinsically preserve these constraints necessitates extrinsic enforcement (i.e. imposing constraints via additional restrictions on discrete spaces), 
which, however, fundamentally conflicts with the stability requirements of the discretization. In fact, a stable finite element discretization must satisfy:
\begin{itemize}
    \item[-] Approximation property: The discrete spaces must uniformly approximate the continuous solution space.
    \item[-] Subcomplex property: The discrete spaces must form a subcomplex of the continuous complex.
    \item[-] Bounded cochain projections: There exist bounded projections from the continuous spaces to the discrete ones, commuting with the differential operators.
\end{itemize}
Extrinsic enforcement of algebraic constraints will violate the subcomplex property. More precisely, for a discrete space $V_h \subset H(\curl,\Omega;\ST)$, the image $\curl V_h$ of the $\curl$ operator is not necessarily a subspace of $H(\div,\Omega;\ST)$. This breaks the condition of the subcomplex property, thereby preventing the existence of bounded cochain projections. Consequently, any stable numerical discretization of \eqref{eq:HodgeEB0} cannot impose the algebraic constraints and will inevitably accumulate constraint violations over time. 

To address this deficiency, the Hessian complex \cite{quenneville2015new} 
\begin{equation}
    P_1\xrightarrow{\subset}H^2(\Omega)\xrightarrow{\hess}H(\curl,\Omega;\ms)\xrightarrow{\curl}H(\div,\Omega;\mt)\xrightarrow{\div}L^2(\Omega;\mr^3)\xrightarrow{}0
\end{equation}
and the $\div\div$ complex \eqref{eq:complex3Db} are proposed.  Both yield the following Hodge wave equation:
\begin{equation}\label{eq:HodgeEB1}
\begin{aligned}
    \begin{cases}
        \dot{\sigma}=\div\div\bm{E},\\
        \dot{\bm{E}}=-\grad\grad\sigma-\sym\curl\bm{B},\\
        \dot{\bm{B}}=\curl\bm{E},
    \end{cases}
\end{aligned}
\end{equation}
with the scalar-valued function $\sigma$, symmetric matrix-valued function $\bm E$ and traceless matrix-valued function $\bm B$.

Compared with \eqref{eq:HodgeEB0}, this system naturally enforces one of the two algebraic constraints (the symmetry of $\bm{E}$ and the traceless-ness of $\bm{B}$) while maintaining equivalence to \eqref{eq:EB}. Nevertheless,  simultaneous enforcement of symmetry and traceless-ness for both $\bm E$ and $\bm B$ remains unachievable due to the subcomplex property.

The conformal Hessian complex \eqref{eq:complex3Dc} is in some sense the most appropriate complex for the linearized Einstein-Bianchi system.
This complex uniquely preserves both symmetry and traceless-ness for $\bm{E}$ and $\bm{B}$ during evolution. \Cref{thm:Hodge=EB} rigorously confirms the well-posedness of the corresponding Hodge wave equation \eqref{eq:HodgeEB2} and its equivalence to the linearized Einstein-Bianchi system, establishing the framework where symmetry and traceless-ness conditions are strictly preserved throughout evolution. 

\begin{proof}[Proof of \Cref{thm:Hodge=EB}] 
    The well-posedness of the Hodge wave equation \eqref{eq:HodgeEB2} can be derived directly from the abstract framework of wave equations in \cite[Theorem 4.6]{quenneville2015new}. Then let $(\sigma,\bm E, \bm B)$ be the unique solution of \eqref{eq:HodgeEB2} with initial conditions $(\sigma_0,\bm E_0,\bm B_0)$. Given the equivalence between the linearized Einstein-Bianchi system \eqref{eq:EB} and the Hodge wave equation \eqref{eq:HodgeEB1} associated to the Hessian complex or the $\div\div$ complex \cite[Proposition 4.19]{quenneville2015new}, it suffices to show that $(\sigma,\bm E, \bm B)$ is also the solution to \eqref{eq:HodgeEB1}. Let $(\tilde{\sigma},\tilde{\bm E}, \tilde{\bm B})$ be the solution of \eqref{eq:HodgeEB1} with initial conditions $(\sigma_0,\bm E_0,\bm B_0)$ as well. Then we can obtain that $\tilde{\sigma}=0$ and $\tilde{\bm E}$ and $\tilde{\bm B}$ are symmetric, traceless and divergence-free for all time. This shows that
    \begin{equation*}
        \hess\tilde{\sigma}=\dev\hess\tilde{\sigma},\quad \curl\tilde{\bm E}=\sym\curl\tilde{\bm E}.
    \end{equation*}
    Hence, $(\tilde{\sigma},\tilde{\bm E}, \tilde{\bm B})$ is also the solution to \eqref{eq:HodgeEB2}. From the well-posedness of the Hodge wave equation \eqref{eq:HodgeEB2}, we can obtain that
     $(\sigma,\bm E, \bm B)=(\tilde{\sigma},\tilde{\bm E}, \tilde{\bm B})$. This concludes the proof.
\end{proof}

\subsection{Traces and trace complexes}
This subsection investigates the inter-element conformity conditions of the $H(\sym\curl;\ST)$ and $H(\div\div;\ST)$ conforming finite element spaces, deriving the traces (i.e. the restriction of functions on lower-dimensional sub-simplices) for each finite element space. We explore the relations between these traces—referred to as the trace complexes—and determine their Sobolev regularity. These results will be used for the subsequent construction of the finite element spaces in both two and three dimensions. Hereafter, we denote the space of symmetric matrices in two dimensions by $\ms_2$ and the space of traceless matrices in two dimensions by $\mt_2$. 

Recall that Green's identity for the $\div\div$ operator reads \cite{Chen2022a}: let $K$ be a tetrahedron of the mesh $\mathcal{T}_h$ and let $\bsi\in C^2(K;\ms)$ and $v\in H^2(K)$, and then we have
\begin{equation}\label{eq:Green_divdiv}
    \begin{aligned}
        (\div\div\bsi,v)_K=&(\bsi,\nabla^2 v)_K-\sum_{F\in\mf(K)}\sum_{e\in\me(F)}(\mn_{\partial F}^T\bsi\mn,v)_e\\
        &-\sum_{F\in\mf(K)}[(\mn^T\bsi\mn,\partial_{\mn}v)_F-(2\div_F(\bsi\mn)+\partial_{\mn}(\mn^T\bsi\mn),v)_F],
    \end{aligned}
\end{equation}
This identity holds for $\bsi\in C^2(K,\ST)$ as well. Hence, we can establish the sufficient $H(\div\div;\ST)$ conformity conditions in the following proposition:
\begin{prop}\label{prop:divdiv_conf}
    A piecewise smooth matrix-valued function $\bsi$ on a tetrahedral grid is in $H(\div\div,\Omega;\ST)$ if its face and edge traces satisfy the following continuity conditions : 
\begin{enumerate}
    \item[-] Across every interior face, both
    \begin{equation}
    \label{eq:tr_divdiv}
        \tr_{\div\div}(\bsi):=2\div_F(\bsi \mn)+\partial_{\mn}(\mn^T \bsi \mn)
    \end{equation} and $\bm n^T \bm \sigma \bm n$ are single-valued.
    \item[-] Across every edge, the components $\mn_i^T\bsi\mn_j$ (for $i,j=1,2$) are single-valued.
\end{enumerate}
\end{prop}
\begin{remark}
    The second condition is sufficient but not necessary. The necessary requirement for the edge integrals in Green's identity \eqref{eq:Green_divdiv} (summed over all the elements) to vanish is the following weaker condition on each interior edge $e$:
    \begin{equation}\label{eq:edge_compat}
        \sum_{F\in\mf(e)}[\![\mn_{\partial F}^T\bsi\mn_F]\!]=0,
    \end{equation}
    where the sum is taken over all the faces $F$ that take $e$ as an edge. However, condition \eqref{eq:edge_compat} is not easy to implement.
\end{remark}

Similarly, let $\bta\in C^1(K;\ST)$ and $\bsi\in H^1(K;\ST)$ for a polyhedron $K$. Then from Green's identity for the $\sym\curl$ operator:
\begin{equation}\label{eq:Green_symcurl}
    (\sym\curl\bta,\bsi)_K=(\bta,\curl\bsi)_K+\sum_{F\in\mf(K)}(\sym(\mn\times\bta),\bsi)_F,
\end{equation}
we can conclude that a three-dimensional piecewise smooth matrix-valued function $\bta$ is in $H(\sym\curl,\Omega;\ST)$ if its face traces $\sym(\mn\times\bta)$ are single-valued across every interior face. This establishes the $H(\sym\curl;\ST)$ conformity conditions.

To use these conformity conditions to construct exact finite element complexes, we need to characterize the relations of these traces. Such a characterization will facilitate the subsequent construction of degrees of freedom of the corresponding finite element spaces of the complex.

Let $K$ be a tetrahedron of the mesh $\mathcal{T}_h$, and let $\bta$ and $\bsi$ be sufficiently smooth symmetric and traceless matrix-valued functions with $\bsi=\sym\curl\bta$. Substitute this into Green's identity for the $\div\div$ operator \eqref{eq:Green_divdiv}:
\begin{equation}\label{eq:Green_divdiv1}
    \begin{aligned}
        (\sym\curl\bta,\nabla^2 v)_K&-\sum_{F\in\mf(K)}\sum_{e\in\me(F)}(\mn_{\partial F}^T\bsi\mn,v)_e\\
        -\sum_{F\in\mf(K)}&(\mn^T\bsi\mn,\partial_{\mn}v)_F+\sum_{F\in\mf(K)}(\tr_{\div\div}\bsi,v)_F=0.
    \end{aligned}
\end{equation}
Next, apply Green's identity for the $\sym\curl$ operator \eqref{eq:Green_symcurl}for tetrahedron $K$:
\begin{equation}\label{eq:Green_symcurl1}
    \begin{aligned}
        &(\sym\curl\bta,\nabla^2 v)_K=(\bta,\curl\nabla^2 v)_K+\sum_{F\in\mf(K)}(\sym(\mn\times\bta),\nabla^2 v)_F \\
        =&\sum_{F\in\mf(K)}(\sym(\Pi_F(\bta\times\mn)\Pi_F),\curl_F\curl_F v)_F\\
        &-\sum_{F\in\mf(K)}(\Pi_F(\bta^T\mn),\curl_F(\partial_{\mn} v))_F.
    \end{aligned}
\end{equation}
Substitute \eqref{eq:Green_symcurl1} into \eqref{eq:Green_divdiv1}, and perform integration by parts for all the faces $F\in\mf(K)$:
\begin{equation*}
\begin{aligned}
    &\sum_{F\in\mf(K)}(\tr_{\div\div}\bsi+\rot_F\rot_F(\sym(\Pi_F(\bta\times\mn)\Pi_F)),v)_F\\
    -&\sum_{F\in\mf(K)}(\mn^T\bsi\mn-\rot_F(\Pi_F(\bta^T\mn)),\partial_{\mn} v)_{F}
    -\sum_{F\in\mf(K)}\sum_{e\in\me(F)} ((\mn^T\bta\mbt_{\partial_F},\partial_{\mn}v)_e+(\mn_{\partial F}^T\bsi\mn,v)_e)\\
    +&\sum_{F\in\mf(K)}\sum_{e\in\me(F)} ((\sym(\Pi_F(\bta\times\mn)\Pi_F)\mbt_{\partial_F},\curl_F v)_e-(\mbt_{\partial F}^T\rot_F(\sym(\Pi_F(\bta\times\mn)\Pi_F)),v)_e)=0.
\end{aligned}
\end{equation*}
With $v$ being an arbitrary smooth function, the above identity yields the following necessary compatibility conditions on each face $F\in\mf(K)$:
\begin{equation}\label{eq:facetrace1}
    \tr_{\div\div}(\bsi)=-\rot_F\rot_F(\sym(\Pi_F(\bta\times\mn))\Pi_F),\quad         \mn^T\bsi\mn=\rot_F(\Pi_F(\bta^T\mn)).
\end{equation}

Note that the quantities $\sym(\Pi_F(\bta\times\mn)\Pi_F)$ and $\Pi_F(\bta^T\mn)$ represent, respectively,  the tangential-tangential components and tangential-normal components of the $H(\sym\curl;\ST)$ traces $\sym(\bta\times\mn)$. (Because $\sym(\bta\times\mn)$ is symmetric and traceless, its normal-normal component is determined by its tangential-tangential components.) Consequently, the conditions \eqref{eq:facetrace1} state precisely how the $H\div\div;\ST)$ traces relate to the $H(\sym\curl;\ST)$ traces. A similar result appears in \cite[Lemma 4.3]{Hu2022d}, where it is obtained by direct computation. Here, we provide a unified detailed derivation.

Following a similar procedure, we can obtain the compatibility conditions between the $H(\sym\curl;\ST)$ traces and the $H^2$ traces: for $\bta=\dev\hess u$,
\begin{equation}\label{eq:facetrace2}
    \sym(\Pi_F(\bta\times\mn)\Pi_F)=-\sym\grad_F\curl_F u,\quad \Pi_F(\bta^T\mn)=\grad_F(\partial_{\mn}u).
\end{equation}
Note that the operator $\sym\grad_F\curl_F$ maps a scalar-valued function to a symmetric and traceless tensor-valued function, since $\Tr(\sym\grad_F\curl_F u)=\div_F\curl_F u=0$. 

From the relations \eqref{eq:facetrace1} and \eqref{eq:facetrace2}, we can see that the relevant differential operators not only commute with the trace maps (i.e. the projections from three‑dimensional tetrahedron to two‑dimensional faces), but also induce complexes on each face of tetrahedron $K$—so‑called trace complexes. The next lemma summarizes these results in the form of two commutative diagrams.
\begin{lemma}\label{lem:face}
    The following two diagrams commute:
    \begin{equation}\label{eq:diagrama}
    \begin{tikzcd}[column sep=huge]
        u \dar[] \rar["\dev\hess "] & \bta \dar[] \rar["\sym\curl"] & \bsi \dar[] \\
        u \rar[" ~~~~-\sym\grad_F\curl_F "] & \sym(\Pi_F(\bta\times\mn)\Pi_F) \rar[" -\rot_F\rot_F  "] & \tr_{\div\div}(\bsi),
    \end{tikzcd}
    \end{equation}
    which leads to the two-dimensional conformal strain complex \eqref{eq:complex2Db} below, and 
    \begin{equation}\label{eq:diagramb}
    \begin{tikzcd}
        u \dar[] \rar["\dev\hess "] & \bta \dar[] \rar["\sym\curl"] & \bsi \dar[] \\
        \partial_{\mn}u \rar["\grad_F"] & \Pi_F(\bta^T\mn) \rar["\rot_F"] & \mn^T\bsi\mn,
    \end{tikzcd}
    \end{equation}
    which leads to the two-dimensional de Rham complex \eqref{eq:complex2Da} below.
    Note that $\tr_{\div\div}$ is defined in \eqref{eq:tr_divdiv}.
\end{lemma}

We need to characterize the function spaces of the traces on face $F$ in both \eqref{eq:diagrama} and \eqref{eq:diagramb}.

Since a three‑dimensional $H^2$ conforming finite element function $u$ is $C^4$-continuous across each vertex and $C^2$-continuous across each edge \cite{hu2023construction}, then, on any face $F$, the trace $u|_F$ belongs to $H^3(F)$ and $\partial_{\mn} u|_F$ belongs to $H^2(F)$. Meanwhile, the $H(\div\div;\ST)$ conformity conditions in \Cref{prop:divdiv_conf} ensure that the normal–normal components $\mn_i^T\bsi\mn_j\ (i,j=1,2)$ are continuous across edges, so $\mn^T\bsi\mn|_F\in H^1(F)$. Hence, from the bottom rows of the diagrams \eqref{eq:diagrama} and \eqref{eq:diagramb}  we can recover the following two-dimensional conformal Hessian complex on $F$:
\begin{align}\label{eq:complex2Db}
    P_{1}^{+}\xrightarrow[]{\subset}H^3(F) \xrightarrow[]{\sym\grad_F\curl_F}H^1(\rot_F\rot_F,F;\ms_2 \cap \mt_2) \xrightarrow[]{\rot_F\rot_F} L^2(F) \rightarrow 0,
\end{align}
and the two-dimensional de Rham complex:
\begin{align}\label{eq:complex2Da}
    \mr\xrightarrow[]{\subset}H^2(F) \xrightarrow[]{ \grad_F}H^1(\rot_F,F;\mr^2) \xrightarrow[]{\rot_F} H^1(F) \rightarrow 0.
\end{align}
Here, $P_{1}^{+} = P_1 \oplus \operatorname{span}\{\bm x^T\bm x\}$ is the kernel of $\sym\grad_F \curl_F$. The matrix-valued space $H^1(\rot_F\rot_F,F;\ms_2 \cap \mt_2)$ is defined as 
\begin{align*}
H^1(\rot_F\rot_F,F;\ms_2\cap\mt_2):=\{\bta\in H^1(F;\ms_2\cap\mt_2)  \colon \rot_F\rot_F\bta\in L^2(F)\},
\end{align*}
and the vector-valued space $H^1(\rot_F,F;\mr^2)$ is defined as 
\begin{align*}
H^1(\rot_F,F;\mr^2):=\{\bta\in H^1(F;\mr^2)  \colon \rot_F\bta\in H^1(F)\}.
\end{align*}
This is, for matrix-valued function $\bta$,
\begin{equation*}
    \sym(\Pi_F(\bta\times\mn)\Pi_F)\in H^1(\rot_F\rot_F,F;\ms_2\cap\mt_2),\quad \Pi_F(\bta^T\mn)\in H^1(\rot_F,F;\mr^2).
\end{equation*}

Both trace complexes \eqref{eq:complex2Db} and \eqref{eq:complex2Da} are exact on a bounded Lipschitz contractible domain. In particular, the conformal strain complex \eqref{eq:complex2Db} is a rotated version of the conformal Hessian complex in two dimensions, and can be derived from the Bernstein-Gelfand-Gelfand machinery \cite{Arnold2021}: 
\begin{equation}\label{eq:defcurlBGG}
    \begin{tikzcd}
    &H^3(F) \arrow{r}{\curl_F} &H^2(F;\mr^2) \arrow{r}{\div_F} &H^1(F)\\
    &H^2(F;\mr^2) \arrow{r}{\sym\grad_F} \arrow[ur, "\mathrm{id}"]&H^1(\rot_F\rot_F,F;\ms_2) \arrow{r}{\rot_F\rot_F} \arrow[ur, "\Tr"] &L^2(F) .
    \end{tikzcd}
\end{equation}
The exactness of the conformal strain complex then follows directly from the diagram chase.

The conforming discretization of the trace complexes \eqref{eq:complex2Da} and \eqref{eq:complex2Db} is constructed in \Cref{sec:trace}, preserving the exactness. These two-dimensional discrete complexes will facilitate
the discretization of the three-dimensional conformal Hessian complex by determining the appropriate degrees of freedom on edges and faces.

\begin{remark} 
    The following conformal strain complex starting from $H^2(F)$ has lower regularity compared with \eqref{eq:complex2Db}:
    \begin{align}\label{eq:complex2Dc}
        P_{1}^{+}\xrightarrow[]{\subset}H^2(F) \xrightarrow[]{\sym\grad_F\curl_F}H(\rot_F\rot_F,F;\ms_2 \cap \mt_2) \xrightarrow[]{\rot_F\rot_F} L^2(F) \rightarrow 0.
    \end{align}
    Here
    \begin{align*}
        H(\rot_F\rot_F,F;\ms_2\cap\mt_2):=\{\bta\in L^2(F;\ms_2\cap\mt_2)  \colon \rot_F\rot_F\bta\in L^2(F)\}.
    \end{align*}
    \eqref{eq:complex2Dc} is also exact on a bounded Lipschitz contractible domain. The finite element conformal strain complex proposed in \Cref{subsec:conformal_strain} is a subcomplex of both \eqref{eq:complex2Db} and \eqref{eq:complex2Dc}. 
\end{remark}

\section{Bubble complexes in three dimensions}\label{sec:bubble}

This section constructs the bubble complex corresponding to the conformal Hessian complex. Building on the bubble complex, we establish its exactness and derive characterizations of the associated $H(\sym\curl;\ST)$ and $H(\div\div;\ST)$ bubble function spaces. These results will subsequently be employed to define the degrees of freedom for the finite element spaces of the conformal Hessian complex and to prove the exactness of the discrete conformal Hessian complex.

The construction of the conformal Hessian bubble complex relies on three key components: (1) the de Rham bubble complex in \eqref{dct_deRham} below and the $\div\div$ bubble complex in \eqref{dct_divdivS+} below with enhanced smoothness, (2) the characterization of the space spanned by the skew-symmetric part of the functions in the $H(\sym\curl;\mt)$ bubble function space, (3) the characterization of the space spanned by the trace of the functions in the $H(\div\div;\ms)$ bubble function space. Note that here the term “trace” refers to the matrix trace (i.e., the sum of the diagonal components of matrices), not the restriction of functions on lower-dimensional sub-simplices. 

We begin by introducing the concept of bubble function spaces with enhanced smoothness. Given a tetrahedron $K$, define
\begin{align*}
    \mathbb{B}_k^{(r_1)}(K;\mathbb{X}):=\{q \in P_k(K;\mathbb{X}) \colon D^{\alpha}q|_{\mv(K)}=0,\, \forall |\alpha| \leq r_1\},
\end{align*}
and given a face $F$, define
\begin{align*}
    \B_k^{(r_1)}(F;\mathbb{X}):=\{q \in P_k(F;\mathbb{X}) \colon D^{\alpha}q|_{\mv(F)}=0,\, \forall ~~|\alpha| \leq r_1\}.
\end{align*}
For finite element spaces with enhanced smoothness at the vertices and on the edges of $K$, define
\begin{align*}
    \B_k^{(r_1,r_2)}(K;\mathbb{X}):=\{q \in P_k(K;\mathbb{X}) \colon D^{\alpha}q|_{\mv(K)}=0,\, \forall |\alpha| \leq r_1,\; D^{\beta}q|_{\me(K)}=0,\, \forall |\beta| \leq r_2\},\\
    \B_k^{(r_1,r_2)}(F;\mathbb{X}):=\{q \in P_k(F;\mathbb{X}) \colon D^{\alpha}q|_{\mv(F)}=0,\, \forall~~ |\alpha| \leq r_1,\; D^{\beta}q|_{\me(F)}=0,\, \forall ~~|\beta| \leq r_2\},
\end{align*}
and for finite element spaces with enhanced smoothness on the faces of $K$,
\begin{align*}
    \B_k^{(r_1,r_2,r_3)}(K;\mathbb{X}):=\{&q \in P_k(K;\mathbb{X}) \colon D^{\alpha}q|_{\mv(K)}=0,\, \forall |\alpha| \leq r_1,\; \\&D^{\beta}q|_{\me(K)}=0,\, \forall |\beta| \leq r_2,\; D^{\gamma}q|_{\mf(K)}=0,\, \forall |\gamma| \leq r_3\}.
\end{align*}
Here, $r_1,r_2,r_3 \in \mathbb{N}$ are non-negative integers. 

Next, we define bubble function spaces corresponding to various differential operators.  Let $\bm r$ denote the index of enhanced smoothness for the bubble function spaces, i.e. $\bm r$ can be of the form $(r_1),\ (r_1,r_2)\ \text{or}\  (r_1,r_2,r_3)$ for non-negative integers $r_1,r_2,r_3$. 
For the $H(\sym\curl;\mathbb{X})$ bubble function space, we have
\begin{align*}
    \B_k^{\bm r}(\sym\curl,K;\mathbb{X}):=\{\bta\in \B_{k}^{\bm r}(K;\mathbb{X}) \colon \sym(\mn\times\bta)|_F=0,\, \forall F\in\mf(K)\},
\end{align*}
with $\mathbb{X}$ being either $\mt$ or $\ST$. 
For the $H(\div\div;\mathbb{X})$ bubble function space, we define
\begin{align*}
    \B_k^{\bm r}(\div\div,K;\mathbb{X}) & := \{\bsi\in \B_{k}^{\bm r}(K;\mathbb{X}) \colon  \mn_i^T\bsi\mn_j|_e=0,\, \forall e\in\me(K),i,j=1,2,\\
    & \mn^T\bsi\mn|_F=2\div_F(\bsi\mn)+\partial_{\mn}(\mn^T\bsi\mn)|_F=0,\,\forall F\in\mf(K)\},
\end{align*}
with $\mathbb{X}$ being $\ms$ or $\ST$.
For the $H(\div;\mr^3)$ bubble function space, we define
\begin{equation*}
    \B_{k}^{\bm r}(\div,K;\mr^3):=\{\bm{u}\in \B_{k}^{\bm r}(K;\mr^3) \colon \bm{u}\cdot\mn|_F=0,\,\forall F\in\mf(K)\}.
\end{equation*}

We now present the conformal Hessian bubble complex. Motivated by the construction of $C^1$ finite element spaces \cite{Zhang2009,hu2023construction,MR275014}, the $H^2$ bubble function space is given as $\B_{k+3}^{(4,2,1)}(K)=b_K^2P_{k-5}(K)$.  Starting with this space, the smoothness conditions of the subsequent bubble function spaces at vertices, edges and faces are determined as follows:
\begin{equation}
\label{eq:continuity}
    (4,2,1) \xrightarrow{\dev\hess} (2,0) \xrightarrow{\sym\curl} (1) \xrightarrow{\div\div} \varnothing.
\end{equation}
This suggests that the $L^2$ conforming finite element space should be of the discontinuous Galerkin type. Following the pattern in \eqref{eq:continuity}, the $H(\sym\curl;\ST)$ bubble function space is $\B_{k+1}^{(2,0)}(\sym\curl,K;\ST)$ and the $H(\div\div;\ST)$ bubble function space is $\B_k^{(1)}(\div\div,K;\ST)$. Therefore, the conformal Hessian bubble complex reads as follows:
\begin{equation}\label{dct_divdivST}
\begin{aligned}
    0 \xrightarrow[]{}  \B_{k+3}^{(4,2,1)}(K) &\xrightarrow[]{\dev\hess} \B_{k+1}^{(2,0)}(\sym\curl,K;\ST) \\&\qquad\xrightarrow[]{\sym\curl} \B_{k}^{(1)}(\div\div,K;\ST) \xrightarrow[]{\div\div} P_{k-2}(K)\cap P_{1}^{+,\perp}.
\end{aligned}
\end{equation}
Here, $P_{k-2}(K) \cap P_{1}^{+,\perp}$ denotes the orthogonal complement space of $P_1^+$ with respect to $P_{k-2}(K)$ under the $L^2$ inner product. 

One main result in this section is the characterizations of the matrix-valued bubble function spaces $\B_{k+1}^{(2,0)}(\sym\curl,K;\ST)$ and $\B_k^{(1)}(\div\div,K;\ST)$ (see \Cref{subsec:charac} below). It is important to note that these characterizations are necessary for applying the bubble function spaces to define the degrees of freedom for the $H(\sym\curl;\ST)$ and $H(\div\div;\ST)$ finite elements. 

Another important result is the exactness of the conformal Hessian bubble complex.
\begin{theorem}[Exactness of finite element bubble complex] \label{thm:conformalbubble_excat}
    Suppose $k\geq 6$. The sequence \eqref{dct_divdivST} forms an exact complex.
\end{theorem}
The exactness of this bubble complex is critical to ensuring the exactness of the finite element conformal Hessian complex.

To prove the exactness of \eqref{dct_divdivST}, we introduce two auxiliary complexes. The first one is the following de Rham bubble complex:
\begin{equation}\label{dct_deRham}
    0 \xrightarrow[]{}  \B_{k+3}^{(4,2,1)}(K) \xrightarrow[]{\grad} \B_{k+2}^{(3,1,0)}(K;\mr^3) \xrightarrow[]{\curl} \B_{k+1}^{(2,0)}(\div,K;\mr^3) \xrightarrow[]{\div} \B_{k}^{(1)}(K)\cap P_0(K)^{\perp}.
\end{equation}
The second one is the following $\div\div$ bubble complex with enhanced smoothness:
\begin{equation}\label{dct_divdivS+}
\begin{aligned} 
    0 \xrightarrow[]{}  \B_{k+2}^{(3,1,0)}(K;\mr^3) &\xrightarrow[]{\dev\grad} \B_{k+1}^{(2,0)}(\sym\curl,K;\mt)\\ &\qquad \xrightarrow[]{\sym\curl} \B_{k}^{(1)}(\div\div,K;\ms) \xrightarrow[]{\div\div} P_{k-2}(K)\cap  P_1(K)^{\perp}.
\end{aligned}
\end{equation}

The proof of \Cref{thm:conformalbubble_excat} relies on the properties of these auxiliary complexes and their connection with the complex of interest \eqref{dct_divdivST}, specifically the following key results:
\begin{itemize}
    \item[-] The bubble complexes \eqref{dct_deRham} and \eqref{dct_divdivS+} are exact.
    \item[-] The space spanned by the skew-symmetric part of the functions in
    $\B_{k+1}^{(2,0)}(\sym\curl,K;\mt)$ is isomorphic to $\B_{k+1}^{(2,0)}(\div,K;\mr^3)$.
    \item[-] The space spanned by the trace of the functions in $\B_{k}^{(1)}(\div\div,K;\ms)$ is isomorphic to $\B_{k}^{(1)}(K)\cap P_0(K)^{\perp}$.
\end{itemize}

In fact, the second and third results above precisely characterize the fundamental difference between the matrix-valued bubble function spaces in \eqref{dct_divdivST} and those used in \eqref{dct_divdivS+}. Thus, by eliminating the corresponding skew-symmetric or trace part from the original bubble function spaces in \eqref{dct_divdivS+}, we can derive the symmetric and traceless bubble function spaces of the conformal Hessian complex \eqref{dct_divdivST}. This is where the auxiliary complexes play a central role. 

In the rest of this section, we first establish the exactness of the two auxiliary bubble complexes \eqref{dct_deRham} and \eqref{dct_divdivS+}, along with the two isomorphism results for the associated bubble function spaces. Then we use these ingredients to show the exactness of the conformal Hessian bubble complex \eqref{dct_divdivST}. Finally, we provide detailed characterizations of the $H(\sym\curl;\ST)$ and $H(\div\div;\ST)$ bubble function spaces.

\subsection{de Rham bubble complex and $\div\div$ bubble complex}
This subsection proves the exactness of the de Rham bubble complex \eqref{dct_deRham} and the $\div\div$ bubble complex \eqref{dct_divdivS+}.

The exactness of the de Rham bubble complex is a direct result from \cite{Chen2024}.
\begin{prop}
Suppose $k\ge 6$. The bubble complex \eqref{dct_deRham} is exact.
\end{prop}
\begin{proof}
    See \cite[Example 5.11]{Chen2024}. 
\end{proof}

The exactness of the $\div\div$ bubble complex \eqref{dct_divdivS+} is based on the following exact complex from \cite{Hu2022d}:
\begin{equation}\label{dct_divdivS}
\begin{aligned}
    0 \xrightarrow[]{}  \B_{k+2}^{(2,1,0)}(K;\mr^3) &\xrightarrow[]{\dev\grad} \B_{k+1}^{(1,0)}(\sym\curl,K;\mt)\\&\qquad \xrightarrow[]{\sym\curl} \B_{k}^{(0)}(\div\div,K;\ms) \xrightarrow[]{\div\div} P_{k-2}(K)\cap P_1(K)^{\perp}.
\end{aligned}
\end{equation} 
In fact, the complex \eqref{dct_divdivS+} feature one-order-higher continuity at the vertices compared with \eqref{dct_divdivS}, representing an enhanced version.

Before we prove the exactness of the $\div\div$ bubble complex \eqref{dct_divdivS+}, we first count the dimension of the $H(\div\div;\ms)$ bubble function space $\mathbb B_{k}^{(1)}(\div\div,K;\ms)$. 

\begin{lemma}\label{lem:dim_divdiv}
    The dimension of $\B_{k}^{(1)}(\div\div,K;\ms)$ is
\begin{equation}\label{eq:dim_divdiv}
    \dim  \B_{k}^{(1)}(\div\div,K;\ms)=k^3+2k^2-3k-28.
 \end{equation}
\end{lemma}

\begin{proof}
To prove \eqref{eq:dim_divdiv}, we proceed in two steps. The first step constructs an $H(\div\div;\ms)$ conforming finite element space. The second step shows that the space $\B_{k}^{(1)}(\div\div,K;\ms)$ is the bubble function space of this $H(\div\div;\,s)$ element, which leads to its dimension.

The degrees of freedom of the $H(\div\div; \ms)$ conforming finite elements for the shape function space $P_{k}(K;\ms)$ are defined as follows:
\begin{subequations}
\begin{align}
    D^{\alpha}\bsi (x)&~~ \text{for all}\; x\in \mv(K), \, |\alpha| \leq 1; \label{divdivS+:1}\\
    (\mn_i^T \bsi \mn_j, q)_e &~~ \text{for all}\; q \in P_{k-4}(e),\, e\in \me(K),\, i,j=1,2; \label{divdivS+:2}\\
    (\mn^T\bsi\mn, q)_F &~~  \text{for all}\; q \in \B_{k}^{(1,0)}(F), \, F\in \mf(K);\label{divdivS+:3} \\
    (2\div_F(\bsi\mn)+\partial_{\mn}(\mn^T\bsi\mn), q)_{F} &~~  \text{for all}\; q \in \B_{k-1}^{(0)}(F), \, F\in \mf(K); \label{divdivS+:4}\\
    (\bsi, \hess q)_{K}&~~  \text{for all}\; q \in P_{k-2}(K); \label{divdivS+:5}\\
    (\bsi, \sym\curl\bta)_{K}&~~ \text{for all}\; \bta \in \B_{k+1}^{(2,0)}(\sym\curl,K;\mt). \label{divdivS+:6}
\end{align}
\end{subequations}

Next, we prove that this set of degrees of freedom is unisolvent with respect to the shape function space $P_k(K;\ms)$. Indeed, a calculation of dimension shows that the number of degrees of  freedom \eqref{divdivS+:1}-\eqref{divdivS+:6} is exactly the dimension of $P_{k}(K;\ms)$, namely,
\begin{align*}
    &96+18(k-3)+2(k-1)(k-2)+2k(k+1)-12+\frac{(k+1)k(k-1)}{6}-4\\&+\frac{4k^3+6k^2-10k-108}{3} -\frac{k^3-k-24}{2}=(k+3)(k+2)(k+1).
\end{align*}
Suppose that the degrees of freedom \eqref{divdivS+:1}-\eqref{divdivS+:6} vanish for some $\bsi \in P_{k}(K;\ms)$, it suffices to prove $\bsi=0$. From \eqref{divdivS+:1}-\eqref{divdivS+:4} we can obtain that $\bsi$ is in $\B_{k}^{(1)}(\div\div,K; \ms)$, which is a subspace of $\B_{k}^{(0)}(\div\div,K; \ms)$. Together with \eqref{divdivS+:5}, an integration by parts shows that $\div\div \bsi=0$. 
By the exactness of the complex \eqref{dct_divdivS}, there exists $\bta \in \B_{k+1}^{(1,0)}(\sym \curl,K;\mt)$ such that $\bsi = \sym\curl \bta$. By \eqref{divdivS+:6}, it suffices to show that $\bta$ has vanishing second order derivatives at vertices.

From the characterization of the $H(\sym\curl;\mt)$ bubble function space $\B_{k+1}^{(1,0)}(\sym \curl,K;\mt)$ in \cite[Remark 4.6]{Hu2022d}, the following decomposition holds:
\begin{equation*}
    \bta = \sum_{i=0}^3 \lambda_j \lambda_l \lambda_m \dev(\bm{p}_i \mn_i^T) + \lambda_0 \lambda_1 \lambda_2 \lambda_3 \bm q,
\end{equation*}
where $\bm p_i \in P_{k-2}(F_i; \mr^3)$, $\bm q \in P_{k-3}(K; \mt)$ and $\{i,j,l,m\}$ is a permutation of $\{0,1,2,3\}$. 
It then suffices to show that $\bm p_i(\bm x_j) = 0$ for any $i \neq j$. 
Considering the first order directional derivatives of $\bsi$ at $\bm{x}_0$, the vanishing degrees of freedom \eqref{divdivS+:1} yield
\begin{equation*}
    \partial_{\bm s} \bsi({\bm{x}_0}) = \sym \sum_{i=1}^3 \lambda_0 \bm p_i(\bm{x}_0) (\nabla (\partial_{\bm s} (\lambda_j \lambda_k)) \times \mn_0)^T = 0,
\end{equation*}
where $\{i,j,k\}=\{1,2,3\}$. Take $\bm s = \mn_i \times \mn_j$ for all $i,j=1,2,3$ and $i<j$, we can derive three matrix-valued identities as follows:
\begin{subequations}
\begin{align}
    \sym (\bm p_1(\bm{x}_0) (\nabla\lambda_2 \times \mn_0)^T + \bm p_2(\bm{x}_0) (\nabla\lambda_1 \times \mn_0)^T)&=0,\label{derivative:1}\\
    \sym (\bm p_1(\bm{x}_0) (\nabla\lambda_3 \times \mn_0)^T + \bm p_3(\bm{x}_0) (\nabla\lambda_1 \times \mn_0)^T)&=0,\label{derivative:2}\\
    \sym (\bm p_2(\bm{x}_0) (\nabla\lambda_3 \times \mn_0)^T + \bm p_3(\bm{x}_0) (\nabla\lambda_2 \times \mn_0)^T)&=0.\label{derivative:3}
\end{align}
\end{subequations}

For $i=1,2,3$, the vector $\nabla \lambda_i \times \mn_0$ is parallel to $\bm x_j-\bm x_k$, where $\{i,j,k\}=\{1,2,3\}$. Hence these three vectors are linearly dependent, while  each two of them are linearly independent. Thus there exists $a_i \neq 0$ such that
\begin{equation*}
    \sum_{i=1}^3 a_i (\nabla \lambda_i \times \mn_0)=0.
\end{equation*}
Then by $a_2 \cdot \eqref{derivative:1} + a_3 \cdot \eqref{derivative:2}$, we have
\begin{equation*}
    \sym( (-a_1 \bm p_1(\bm{x}_0) + a_2 \bm p_2(\bm{x}_0) + a_3 \bm p_3(\bm{x}_0)) (\nabla \lambda_1 \times \mn_0)^T )=0.
\end{equation*}
Since $\nabla \lambda_1 \times \mn_0 \neq 0$, then 
\begin{equation*}
    -a_1 \bm p_1(\bm{x}_0) + a_2 \bm p_2(\bm{x}_0) + a_3 \bm p_3(\bm{x}_0) = 0.
\end{equation*}
Similarly, from \eqref{derivative:1}-\eqref{derivative:3} we can derive that 
\begin{align*}
    a_1 \bm p_1(\bm{x}_0) - a_2 \bm p_2(\bm{x}_0) + a_3 \bm p_3(\bm{x}_0) = 0, \\
    a_1 \bm p_1(\bm{x}_0) + a_2 \bm p_2(\bm{x}_0) - a_3 \bm p_3(\bm{x}_0) = 0.
\end{align*}
Then $a_i \neq 0$ for all $i=1,2,3$ leads to $\bm p_i(\bm{x}_0)=0$ for $i=1,2,3$. Similarly, it holds that $\bm p_i(\bm{x}_j) = 0$ for all $i,j=0,1,2,3$ and $i \neq j$.

Therefore, the second order derivatives of $\bta$ vanish at all vertices of $K$ and $\bta \in \B_{k+1}^{(2,0)}(\sym\curl, K; \mt)$. This plus the degrees of freedom \eqref{divdivS+:6} shows that $\bsi=0$, which proves the unisolvency.

Finally, from the definition of  $\B_{k}^{(1)}(\div\div,K;\ms)$, this space is exactly the subspace of $P_{k}(K;\ms)$ with vanishing degrees of freedom \eqref{divdivS+:1}-\eqref{divdivS+:4}. Thus the dimension of $\B_{k}^{(1)}(\div\div,K;\ms)$ is
\begin{equation*}
    \dim  \B_{k}^{(1)}(\div\div,K;\ms)=k^3+2k^2-3k-28.
 \end{equation*}
This completes the proof.
\end{proof}

The following proposition proves the exactness of the $\div\div$ bubble complex \eqref{dct_divdivS+} with enhanced smoothness.
\begin{prop}
Suppose $k\ge 6$. The bubble complex \eqref{dct_divdivS+} is exact.
\end{prop} 
\begin{proof}

We prove the exactness from left to right. 
First, we show that $$\B_{k+1}^{(2,0)}(\sym \curl,K;\mt) \cap \ker(\sym \curl) = \dev \grad \B_{k+2}^{(3,1,0)}(K; \mr^3),$$ namely, if $\sym \curl \bta = 0$ and $\bta \in \B_{k+1}^{(2,0)}(\sym \curl,K;\mt)$, there exists a $\bm v \in \B_{k+2}^{(3,1,0)}(K; \mr^3)$ such that $ \bta = \dev \grad \bm v$.
  
   By the exactness of the $\div\div$ bubble complex \eqref{dct_divdivS}, there exists $\bm v = (v_1, v_2, v_3)^T \in \B_{k+2}^{(2,1,0)}(K;\mr^3)$ such that $\bta = \dev \grad \bm v$. Since the second-order derivatives of $\bta$ vanish at the vertices of tetrahedron $K$, $\div(\bta^T)=\frac{2}{3}\grad(\div \bm v)$ has vanishing first-order derivatives at the vertices. Hence, the second order derivatives of $\div\bm v$ vanish at the vertices, and so is $\grad \bm v=\dev\grad \bm v + \frac{1}{3}(\div \bm v) \bm I$. This concludes that the third-order derivatives of $\bm v$ vanish at the vertices and that $\bm v \in \B_{k+2}^{(3,1,0)}(K; \mr^3)$.
    
Next, we show that $$\B_{k}^{(1)}(\div\div,K;\ms) \cap \ker(\div\div) = \sym\curl \B_{k+1}^{(2,0)}(\sym \curl,K;\mt),$$ namely, if $\div\div \bsi = 0$ and $\bsi \in \B_{k}^{(1)}(\div\div,K;\ms)$, there exists $\bta \in \B_{k+1}^{(2,0)}(\sym \curl,K;\mt)$ such that $\bsi = \sym\curl \bta$. 
This follows from the proof of unisolvency for the degrees of freedom of the $H(\div\div;\mathbb{S})$ conforming finite elements in \Cref{lem:dim_divdiv}.

Finally, we show that  $P_{k-2}(K)\cap P_{1}(K)^{\perp} = \div\div \B_{k}^{(1)}(\div\div,K; \ms)$, namely, if $q \in P_{k-2}(K)$ and $q \perp P_{1}(K)$, there exists $\bsi \in \B_{k}^{(1)}(\div\div,K; \ms)$ such that $q = \div\div \bsi$. 

    From the integration by parts, we can obtain that $\div\div \B_{k}^{(1)}(\div\div,K; \ms)$ is orthogonal to $P_1(K)$, and hence
    \begin{equation*}
        \div\div \B_{k}^{(1)}(\div\div,K; \ms) \subset P_{k-2}(K)\cap P_{1}(K)^{\perp}.
    \end{equation*}
    Moreover, it holds that
    \begin{align*}
        &\dim  (\div\div \B_{k}^{(1)}(\div\div,K; \ms)) \\
        = & \dim  \B_{k}^{(1)}(\div\div,K; \ms) - \dim  (\B_{k}^{(1)}(\div\div,K;\ms) \cap \ker(\div\div)) \\
        = & \dim  \B_{k}^{(1)}(\div\div,K; \ms) - \dim  (\sym\curl \B_{k+1}^{(2,0)}(\sym \curl,K;\mt)) \\
        = & \dim  \B_{k}^{(1)}(\div\div,K; \ms) - (\dim  \B_{k+1}^{(2,0)}(\sym \curl,K;\mt) - \dim  \B_{k+2}^{(3,1,0)}(K;\mr^3)) \\
        = & k^3+2k^2-3k-28 - \frac{4k^3+6k^2-10k-108}{3} + \frac{k^3-k-24}{2} \\
        = & \frac{k^3-k-24}{6} = \dim  (P_{k-2}(K)\cap P_{1}(K)^{\perp}).
    \end{align*}
    This completes the proof.
\end{proof}

\subsection{Equivalence between bubble function spaces}
This subsection proves the equivalence between: (1) the space spanned by the skew-symmetric part of the matrix-valued functions in$\B_{k+1}^{(2,0)}(\sym\curl,K;\mt)$ and the space $\B_{k+1}^{(2,0)}(\div,K;\mr^3)$, (2) the space spanned by the trace of the matrix-valued functions in $\B_k^{(1)}(\div\div,K;\ms)$ and the space $\B_k^{(1)}(K)\cap P_0(K)^{\perp}$.

\begin{prop}\label{lem:vskw_symcurl} It holds that
\begin{equation}
\label{eq:vskw_symcurl_basis}
   \vskw \B_{k+1}^{(2,0)}(\sym\curl,K;\mt) = \sum_{F\in\mf(K)} b_F \B_{k-2}^{(0)}(F)\mathbb{V}_{F}\oplus b_K  P_{k-3}(K;\mr^3),
\end{equation}
where $\mathbb{V}_{F}={\rm span}\{\mbt_1,\mbt_2\}$ is the space of tangential vectors of $F$. Consequently, we have 
\begin{align}\label{eq:vskw_symcurl=div}
    \vskw \B_{k+1}^{(2,0)}(\sym\curl,K;\mt) = \B_{k+1}^{(2,0)}(\div,K;\mr^3).
\end{align}
\end{prop} 
\begin{proof}
    We first show the direct sum decomposition in \eqref{eq:vskw_symcurl_basis}.
    The right-hand side is a direct sum, because the functions in the intersection of the two subspaces are zero on all faces and hence the intersection of two summands is trivial.

    It follows from a similar argument of \cite[Remark 4.6]{Hu2022d} that the space $\B_{k+1}^{(2,0)}(\sym\curl,K;\mt)$ has the following direct sum decomposition
    \begin{align*}
        \B_{k+1}^{(2,0)}(\sym\curl,K;\mt)=\sum_{F\in\mf(K)}b_F \B_{k-2}^{(0)}(F) \mt_{F}\oplus b_K P_{k-3}(K;\mt),
    \end{align*}
    where $\mt_{F}={\rm span}\{\mbt_1\mn^T,\mbt_2\mn^T,\mn\mn^T-\frac{1}{3}\mI\}$. 
    Hence, to prove \eqref{eq:vskw_symcurl_basis} is to show that
    $\vskw\mt_F=\mathbb{V}_F$. Note that the skew-symmetric part of the identity matrix vanishes. Then we have 
    \begin{equation*}
        \vskw (\mn\mn^T-\frac{1}{3}\bm I)=\vskw(\mn\mn^T)=\frac{1}{2}\mn\times\mn=0.
    \end{equation*}
    Besides,
        \begin{align*}
            \vskw(\mbt_1\mn^T)=\frac{1}{2}\mbt_1\times\mn=-\frac{1}{2}\mbt_2,\\
            \vskw(\mbt_2\mn^T)=\frac{1}{2}\mbt_2\times\mn=\frac{1}{2}\mbt_1.
        \end{align*}
     Hence $\vskw\mt_F= \mathbb{V}_F$ and the direct sum decomposition in \eqref{eq:vskw_symcurl_basis} holds.
    
    The characterization of $\vskw\B_{k+1}^{(2,0)}(\sym\curl,K;\mt)$ shows that the normal component of the functions in $\vskw\B_{k+1}^{(2,0)}(\sym\curl,K;\mt)$ vanishes on the faces of $K$, and the functions themselves vanish on the edges, their values, first and second order derivatives vanish at the vertices. Hence $\vskw\B_{k+1}^{(2,0)}(\sym\curl,K;\mt)$ is a subspace of $\B^{(2,0)}(\div,K;\mr^3)$. Then by \eqref{eq:vskw_symcurl_basis} we have
    \begin{equation*}
        \dim  (\vskw \B_{k+1}^{(2,0)}(\sym \curl,K; \mt)) = \frac{k^3+5k^2-6k-48}{2} = \dim  \B_{k+1}^{(2,0)}(\div,K;\mr^3).
    \end{equation*}
    This completes the proof of \eqref{eq:vskw_symcurl=div}.
\end{proof}

\begin{prop}\label{lem:tr_divdiv}
    It holds that
    \begin{equation}
        \Tr\B_k^{(1)}(\div\div,K;\ms)=\B_k^{(1)}(K)\cap P_0(K)^{\perp}.
    \end{equation}
\end{prop}
\begin{proof}
    For any $q\in B_k^{(1)}(K)\cap P_0(K)^{\perp}$, by the exactness of \eqref{dct_deRham}, there exists $\bm u \in \B_{k+1}^{(2,0)}(\div,K;\mr^3)$ such that $\div \bm u = q$. By the equivalence in \Cref{lem:vskw_symcurl},  there exists $\bta \in \B_{k+1}^{(2,0)}(\sym \curl,K; \mt)$ such that $\bm u = 2\vskw \bta$. Let $\bsi =\sym\curl \bta$, we obtain that
    \begin{equation*}        \Tr\bsi=\Tr(\sym\curl\bta)=\div(2\vskw\bta)=\div\bm u=q.
    \end{equation*}
    This concludes the proof.
\end{proof}
\begin{remark}
    The equivalence in \Cref{lem:vskw_symcurl} and \Cref{lem:tr_divdiv} can be derived by the following BGG-type diagram, clarifying the connections between the two bubble complexes. 
\begin{equation*}\label{dct_BGG}
    \begin{tikzcd}
    \B_{k+3}^{(4,2,1)}(K) \arrow{r}{\grad} &\B_{k+2}^{(3,1,0)}(K;\mathbb{R}^3) \arrow{r}{\curl} &\B_{k+1}^{(2,0)}(\div,K;\mathbb{R}^3) \arrow{r}{\div} &\B_{k}^{(1)}(K)\cap  P_{0}(K)^{\perp}\\
    \B_{k+2}^{(3,1,0)}(K;\mathbb{R}^3)\arrow{r}{\dev\grad} \arrow[ur, "\mathrm{id}"]&\B_{k+1}^{(2,0)}(\sym\curl,K;\mt) \arrow{r}{\sym\curl} \arrow[ur, "2\mathrm{vskw}"]&\B_{k}^{(1)}(\div\div,K;\ms)\arrow{r}{\div\div} \arrow[ur, "\Tr"] &P_{k-2}(K)\cap P_1(K)^{\perp}.
    \end{tikzcd}
\end{equation*}
The equivalence implies that the operators $2\vskw$ and $\Tr$ are onto.
\end{remark}

\subsection{Exactness of the conformal Hessian bubble complex}\label{subsec:conformalbubble_exact}
This subsection proves \Cref{thm:conformalbubble_excat}. 
\begin{proof}
    Suppose $\bta\in\B_{k+1}^{(2,0)}(\sym\curl,K;\ST)$. By \cite[Lemma 5.8]{Chen2022a}, the following identity holds for $e\in\me(K)$:
    \begin{align*}
        \mn_i^T\sym\curl\bta\mn_j = 0,\;i,j=1,2.
    \end{align*}
    This and the two commutative diagrams in \Cref{lem:face} show the inclusions $$\dev\hess \B_{k+3}^{(4,2,1)}(K)\subset \B_{k+1}^{(2,0)}(\sym\curl,K;\ST)$$ and $$\sym\curl \B_{k+1}^{(2,0)}(\sym\curl,K;\ST)\subset \B_{k}^{(1)}(\div\div,K;\ST)$$ hold. Therefore, the sequence \eqref{dct_divdivS+} is a complex.
    
    It remains to prove the exactness, which is equivalent to show the following three statements:
    \begin{enumerate}
        \item $P_{k-2}(K)\cap P_{1}^{+,\perp} = \div\div \B_{k}^{(1)}(\div\div,K; \ST)$, 
        \item $\B_{k}^{(1)}(\div\div,K;\ST) \cap \ker(\div\div) = \sym\curl \B_{k+1}^{(2,0)}(\sym \curl,K;\ST)$,
        \item $\B_{k+1}^{(2,0)}(\sym \curl,K;\ST) \cap \ker(\sym \curl) = \dev \hess \B_{k+3}^{(4,2,1)}(K)$.
    \end{enumerate}

    First, by the exactness of \eqref{dct_divdivS+}, there exists $\bsi \in \B_{k}^{(1)}(\div\div,K; \ms)$ such that $\div\div \bsi = q$ for any $q\in P_{k-2}(K)\cap P_{1}^{+,\perp}$. It follows from  $q \perp P_{1}^{+}$ that
    \begin{equation*}
        0 = (\div\div \bsi, \frac{x^2+y^2+z^2}{2})_K = (\bsi, \bm I)_K = (\Tr\bsi, 1)_K.
    \end{equation*}
    That is, $\Tr \bsi \in \B_{k}^{(1)}(K)\cap P_0(K)^{\perp}$. From \Cref{lem:tr_divdiv}, there exists $\tilde{\bsi} \in \B_{k}^{(1)}(\div\div,K;\ms)$, such that $\Tr\tilde{\bsi}=\Tr\bsi$ and $\tilde{\bsi}=\sym\curl\bta$ for some $\bta\in \B_{k+1}^{(2,0)}(\sym\curl,K;\mt)$.
    Therefore, $\bm \sigma - \tilde{\bm \sigma}$ is symmetric and traceless and $\div\div (\bm \sigma - \tilde{\bm \sigma}) = \div\div \bsi = q$. This proves the first statement.

    Second, again by the exactness of \eqref{dct_divdivS+}, there exists $\bta \in \B_{k+1}^{(2,0)}(\sym \curl,K;\mt)$ such that $\bsi = \sym\curl \bta$ for any $\bsi\in \B_k^{(1)}(\div\div,K;\ST)$ with $\div\div\bsi=0$. Notice that $\bsi$ is traceless. Therefore, $
        0=\Tr \bsi = \Tr (\sym\curl \bta) = 2 \div (\vskw \bta).$ 
    Let $\bm u = 2 \vskw \bta$, then $\bm u$ is divergence-free.  It follows from \Cref{lem:vskw_symcurl} that $\bm u $ is in $\B_{k+1}^{(2,0)}(\div, K; \mr^3) \cap \ker(\div)$. According to the exactness of \eqref{dct_deRham}, there exists $\bm v \in \B_{k+2}^{(3,1,0)}(K;\mr^3)$ such that $\bm u = \curl \bm v$. Let $\tilde{\bta} := \bta - \dev\grad \bm v$, we obtain that
    \begin{align*}
        2 \vskw \tilde{\bta} = \bm u - 2 \vskw(\dev\grad \bm v) = \bm u - \curl \bm v =0.
    \end{align*}
    Therefore, $\tilde{\bta}\in\B_{k+1}^{(2,0)}(\sym\curl,K;\ST)$. Moreover, it holds that $\sym\curl \tilde{\bta} = \sym\curl \bta = \bsi$. This proves the second statement.
    
    Last, by the exactness of the $\div\div$ bubble complex \eqref{dct_divdivS+}, there exists $\bm v \in \B_{k+2}^{(3,1,0)}(K; \mr^3)$ such that $\bta = \dev\grad \bm v$ for any $\bta\in \B_{k+1}^{(2,0)}(\sym\curl,K;\ST)$ with $\sym\curl\bta=0$. Since $\bta$ is also symmetric, we can obtain that 
    \begin{equation*}
        0 = 2 \vskw \bta = 2 \vskw (\dev\grad \bm v) = \curl \bm v.
    \end{equation*}
    Therefore, $\bm v$ is curl free. By the exactness of the de Rham bubble complex \eqref{dct_deRham}, there exists $u \in \B_{k+3}^{(4,2,1)}(K)$ such that $\bm v = \grad u$. Hence, $\bta = \dev\grad \bm v = \dev\hess u.$ This proves the third statement.
\end{proof}

\subsection{Characterizations of the bubble function spaces}\label{subsec:charac}
This subsection presents an explicit characterization of the $H(\sym\curl;\ST)$ bubble function space and an implicit characterization of the $H(\div\div;\ST)$ bubble function space.

The basis of the $H(\sym\curl;\ST)$ bubble function space $\B_{k+1}^{(2,0)}(\sym\curl,K;\ST)$ is constructed in the following lemma. 
\begin{lemma}\label{lem:symcurl_basis}
$\B_{k+1}^{(2,0)}(\sym\curl,K;\ST)$ has the following direct sum decomposition:
    \begin{align*}
        \B_{k+1}^{(2,0)}(\sym\curl,K;\ST)=\sum_{F\in\mf(K)}b_F \B_{k-2}^{(0)}(F)(\mn\mn^T-\frac{1}{3}\mI)\oplus  b_K P_{k-3}(K;\ST).
    \end{align*}
\end{lemma}
\begin{proof}
    It has been shown in \Cref{lem:vskw_symcurl} that the basis of $\B_{k+1}^{(2,0)}(\sym\curl,K;\mt)$ can be characterized as follows:
    \begin{align*}
        \B_{k+1}^{(2,0)}(\sym\curl,K;\mt)=\sum_{F\in\mf(K)}b_F \B_{k-2}^{(0)}(F) \mt_{F}\oplus b_k P_{k-3}(K;\mt),
    \end{align*}
    where $\mt_{F}={\rm span}\{\mbt_1\mn^T,\mbt_2\mn^T,\mn\mn^T-\frac{1}{3}\mI\}$. Since $\B_{k+1}^{(2,0)}(\sym\curl,K;\ST)$ is the subspace of $\B_{k+1}^{(2,0)}(\sym\curl,K;\mt)$ containing all symmetric functions, it suffices to show that $\ST_F={\rm span}\{\mn\mn^T-\frac{1}{3}\mI$\}. 
    
    This is accomplished by direct verification. Suppose that there exist $a,b,c\in\mr$ such that $a\mbt_1\mn^T+b\mbt_2\mn^T+c(\mn\mn^T-\frac{1}{3}\mI)$ is symmetric, then the skew-symmetric part vanishes, namely,
    \begin{align*}
        a(\mbt_1\mn^T-\mn\mbt_1^T)+b(\mbt_2\mn^T-\mn\mbt_2)=0.
    \end{align*}
    Multiplying both sides of the identity by $\mbt_1$ and $\mbt_2$ from the right yields that $a=0$ and $b=0$ respectively. This concludes the proof.
\end{proof}

To characterize the $H(\div\div;\ST)$ conforming bubble function space $\B_{k}^{(1)}(\div\div,K;\ST)$, we first decompose it into two subspaces:
\begin{enumerate}
\item[-] The kernel subspace $\B_{k}^{(1)}(\div\div,K;\ST) \cap \ker(\div\div)$,
\item[-] The quotient space $\B_{k}^{(1)}(\div\div,K;\ST) / \ker(\div\div)$.
\end{enumerate}
From the exactness of the conformal Hessian bubble complex in \Cref{thm:conformalbubble_excat}, the kernel subspace is identified as the image of the preceding bubble function space under the $\sym\curl$ operator:
\begin{equation*}
\B_{k}^{(1)}(\div\div,K;\ST) \cap \ker(\div\div) = \sym\curl\B_{k+1}^{(2,0)}(\sym\curl,K;\ST).
\end{equation*}
Then the basis of the kernel subspace is derived by applying $\sym\curl$ to the basis of $\B_{k+1}^{(2,0)}(\sym\curl,K;\ST)$ as characterized in \Cref{lem:symcurl_basis}.

Furthermore, the exactness of the bubble complex establishes a canonical isomorphism for the quotient space:
\begin{equation*}
    \B_{k}^{(1)}(\div\div,K;\ST) / \ker(\div\div) \cong \div\div\B_k^{(1)}(\div\div,K;\ST)=P_{k-2}(K) \cap P_1^{+,\perp}.
\end{equation*}
From this isomorphism, a dual basis for the quotient space is given by:
\begin{equation*}
    \{( \div\div\cdot,q) \colon \forall q \in P_{k-2}(K) \cap P_1^{+,\perp} \}.
\end{equation*}
These dual basis functionals will be used to define the degrees of freedom for the $H(\div\div;\ST)$ conforming finite element space in \Cref{subsec:divdiv}.

\section{Discrete complexes in two dimensions}\label{sec:trace}

This section provides a conforming discretization of the two-dimensional conformal strain complex \eqref{eq:complex2Db} and the two-dimensional de Rham complex \eqref{eq:complex2Da}. The main focus is on the two-dimensional bubble complexes, which serve as building blocks in the construction of the three-dimensional discrete complexes. 

We begin by defining the associated two-dimensional bubble function spaces. Let $\bm r$ denote the index of enhanced smoothness, i.e. $\bm r$ can be of the form $(r_1)$ or $(r_1,r_2)$ for non-negative integers $r_1,r_2$. The $H^1(\rot_F\rot_F;\mathbb{X})$ bubble function space is defined as
\begin{equation*}
    \B_{k}^{\bm r}(\rot_F\rot_F,F;\mathbb{X}):=\{\bta \in \B_{k}^{\bm r}(F;\mathbb{X})  \colon \mbt^T\bta\mbt|_e = -\partial_{\mbt}(\mn^T\bta\mbt)+\mbt^T\rot_F\bta|_e = 0,\,\forall e\in\me(F)\},
\end{equation*}
with $\mathbb{X}$ being either $\ms_2$ or $\ms_2\cap\mt_2$.

For the $H^1(\rot_F;\mr^2)$ bubble function space, we introduce an enhanced smoothness index of the form $(r_1,+)$ or $(r_1,r_2+)$ with $r_1,r_2\in \mathbb{N}$. The corresponding spaces are defined as
\begin{align*}
    \B_{k}^{(r_1,+)}(\rot_F,F;\mr^2)&:=\{\bm{u} \in \B_{k}^{(r_1)}(F;\mr^2)  \colon \bm{u}\cdot\mbt|_e = \rot_F\bm{u}|_e = 0,\,\forall e\in\me(F)\},\\
    \B_{k}^{(r_1,r_2+)}(\rot_F,F;\mr^2)&:=\{\bm{u} \in \B_{k}^{(r_1,r_2)}(F;\mr^2)  \colon \bm{u}\cdot\mbt|_e = \rot_F\bm{u}|_e = 0,\,\forall e\in\me(F)\}.
\end{align*} 
Here, the "+" notation denotes the $H^1$ regularity of $\rot_F \bm{u}$.

We then introduce two families of bubble complexes, which are used for the discretization of \eqref{eq:complex2Db} and \eqref{eq:complex2Da}.
The following bubble complexes are exact, as shown in \cite{hu2023construction, Hu2022d}:
\begin{enumerate}
\item The de Rham bubble complex in two dimensions:\begin{equation}\label{derham:B}
    0 \xrightarrow[]{\subset}  \B_{k+2}^{(2,0)}(F) \xrightarrow[]{\grad_F} \B_{k+1}^{(1,+)}(\rot_F,F;\mr^2) \xrightarrow[]{\rot_F} \B_{k}^{(0,0)}(F) \cap P_0(F)^{\perp}\xrightarrow[]{}0.
\end{equation}
\item The strain bubble complex in two dimensions: \begin{equation}\label{strain:B}
\begin{aligned}
    0\xrightarrow[]{\subset} \B_{k+2}^{(2,0)}(F;\mr^2) \xrightarrow[]{\sym\grad_F} &\B_{k+1}^{(1)}(\rot_F\rot_F,F;\ms_2)\\&\qquad\xrightarrow[]{\rot_F\rot_F} P_{k-1}(F) \cap P_1(F)^{\perp}\xrightarrow[]{}0.
\end{aligned}
\end{equation}
\end{enumerate}

\subsection{Finite element conformal strain complex in two dimensions}\label{subsec:conformal_strain}
This subsection constructs the finite elements with respect to the conformal strain complex \eqref{eq:complex2Db}.

The shape function space of the $H^3$ conforming finite element is $P_{k+3}(F)$ with $k\geq 6$. For $p \in P_{k+3}(F)$, the following degrees of freedom coincide with \cite{hu2023construction,MR282540}:
\begin{subequations}
\begin{align}
    D^{\alpha}p(x) &~~ \text{for all}\; x \in \mv(F), \, |\alpha| \leq 4; \label{defcurl1:1}\\
    (p,q)_e &~~ \text{for all}\; q \in P_{k-7}(e), \, e \in \me(F); \label{defcurl1:2}\\
    (\frac{\partial p}{\partial \mn}, q)_e &~~ \text{for all}\; q \in P_{k-6}(e), \, e\in\me(F); \label{defcurl1:3}\\
    (\frac{\partial^2 p}{\partial \mn^2}, q)_e &~~ \text{for all}\; q \in P_{k-5}(e), \, e\in\me(F); \label{defcurl1:4}\\
    (p,q)_F &~~ \text{for all}\; q\in P_{k-6}(F).\label{defcurl1:5}
\end{align}
\end{subequations}

The shape function space of the $H^1(\rot_F\rot_F;\ms_2\cap\mt_2)$ conforming finite element is $P_{k+1}(F;\ms_2 \cap \mt_2)$ with $k\geq 4$. For $\bm \tau \in P_{k+1}(F; \mathbb S_2 \cap \mathbb T_2)$, the degrees of freedom are given as follows:
\begin{subequations}
\begin{align}
    D^{\alpha}\bta(x) &~~ \text{for all}\; x \in \mv(F), \, |\alpha| \leq 2; \label{defcurl2:1}\\
    (\bta,\bm{\xi})_e &~~ \text{for all}\; \bm{\xi} \in P_{k-5}(e;\ms_2 \cap \mt_2), \, e \in \me(F); \label{defcurl2:2}\\
    (-\partial_{\mbt}(\mn^T \bta \mbt)+\mbt^T \rot_F \bta, q)_e &~~ \text{for all}\; q \in P_{k-4}(e), \, e\in\me(F); \label{defcurl2:3}\\
    (\bta,\bm{\xi})_F &~~ \text{for all}\; \bm{\xi}\in \dev \curl_F \curl_F \B_{k-1}^{(0)}(F) \label{defcurl2:4}\\ 
     & \phantom{\text{for all}\; \bm{\xi}\in \dev \curl_F}+\sym\grad_F \curl_F \B_{k+3}^{(4,2)}(F). \notag
\end{align}
\end{subequations}

 An integration by parts shows that the two spaces in the degrees of freedom \eqref{defcurl2:4} are orthogonal to each other under the $L^2$ inner product, indicating that they form a direct sum.
\begin{lemma}\label{uni:rotrot}
    The degrees of freedom \eqref{defcurl2:1}-\eqref{defcurl2:4} are unisolvent for $P_{k+1}(F;\ms_2 \cap \mt_2)$.
\end{lemma}

\begin{proof}
    It can be proved that
    $
        \B_{k+3}^{(4,2)}(F)=(\lambda_1\lambda_2\lambda_3)^3 P_{k-6}(F).
    $ 
    Therefore, the number of degrees of freedom is equal to the dimension of $P_{k+1}(F;\mr^2)$, namely, 
    \begin{equation*}
        36+6(k-4)+3(k-3)+\frac{k(k+1)}{2}-7+\frac{(k-4)(k-5)}{2}=(k+2)(k+3).
    \end{equation*}
    Suppose \eqref{defcurl2:1}-\eqref{defcurl2:4} vanish for some $\bta \in P_{k+1}(F;\ms_2 \cap \mt_2)$, then $\bta \in \B_{k+1}^{(1)}(\rot_F \rot_F,F;\ms_2)$. Note that \eqref{defcurl2:4} implies that  $\rot_F\rot_F \bta=0$. From the exactness of the bubble complex in \eqref{strain:B}, we can derive that 
    \begin{equation*}
        \bta=\sym\grad_F \bm p
    \end{equation*}
    for some $\bm p = (p_1, p_2)^T \in \B_{k+2}^{(2,0)}(F;\mr^2)$. 
    
   From the degrees of freedom \eqref{defcurl2:1}, the second-order derivatives of $\sym\grad_F\bm{p}=\bta$ vanish at the vertices. Namely, the second-order derivatives of $\partial_x p_1$, $\partial_y p_2$ and $\partial_y p_1 + \partial_x p_2$ vanish at the vertices of $F$. To show that $\bm{p}$ has vanishing third-order derivatives at the vertices, it suffices to show that $\partial_y^3p_1(\bm{a})=\partial_x^3p_2(\bm{a})=0$ for any $\bm{a}\in\mv(F)$. An algebraic elimination yields that
    \begin{equation*}
       \partial_y^3 p_1(\bm{a})=\partial_y^2(\partial_y p_1 + \partial_x p_2)(\bm{a})-\partial_x \partial_y(\partial_y p_2)(\bm{a})=0
    \end{equation*}
    for any $\bm{a} \in \mv(F)$. Similarly, $\partial_x^3 p_2$ vanishes at these vertices.  

    Since $\bm p$ vanishes on $\partial F$, we can rewrite $\bm p = \lambda_1 \lambda_2 \lambda_3 \tilde{\bm p}$ for some $\tilde{\bm p} \in P_{k-1}(F;\mr^2)$. The degrees of freedom \eqref{defcurl2:2} imply that $\bta$ vanishes on $\partial F$, so $\tilde{\bm p} = 0$ on $\partial F$. Hence, the first-order derivatives of $\bm p$ vanish on $\partial F$, and thus $\bm p \in \B_{k+2}^{(3,1)}(F; \mr^2)$.
    
    Moreover, we have $
        \div_F \bm p=\Tr \bta =0.$ 
    From the exactness of the de Rham bubble complex \cite[Corollary 4.2]{Chen2022d}, there exists $q \in \B_{k+3}^{(4,2)}(F)$ such that $
        \bm p=\curl_F q.$
    Thus the degrees of freedom \eqref{defcurl2:4} leads to $\bta=\sym\grad_F\curl_F q=0$. This concludes the proof.
\end{proof}

The shape function  space of the $L^2$ element is $P_{k-1}(F)$ with $k\geq 2$. For $p \in P_{k-1}(F)$, the degrees of freedom are given as follows:
\begin{subequations}
\begin{align}
    p(x) &~~ \text{for all}\; x \in \mv(F); \label{defcurl3:1}\\
    (p,q)_F &~~ \text{for all}\; q \in \B_{k-1}^{(0)}(F).\label{defcurl3:2}
\end{align}
\end{subequations}

Note that $\B_{k+3}^{(4,2)}(F)$ is exactly the $H^3$ bubble function space of $P_{k+3}(F)$ with vanishing degrees of freedom \eqref{defcurl1:1}-\eqref{defcurl1:4}. Define the $H(\rot_F \rot_F;\ms_2\cap\mt_2)$ bubble function space $\B_{k+1}^{(2,0)}(\rot_F\rot_F,F;\ms_2 \cap \mt_2)$ as
\begin{align*}
    \{\bta\in \B_{k+1}^{(2,0)}(F;\ms_2 \cap \mt_2)  \colon -\partial_{\mbt}(\mn^T \bta \mbt)+\mbt^T \rot_F \bta|_e=0,\, \forall e\in\me(F)\},    
\end{align*}
which is the subspace of $P_{k+1}(F;\ms_2\cap\mt_2)$ with vanishing degrees of freedom \eqref{defcurl2:1}-\eqref{defcurl2:3}. These bubble function spaces form an exact complex, see \Cref{lem:conformal_strain_exact} below.

\begin{lemma}\label{lem:conformal_strain_exact}
    Suppose $k\geq 6$. The complex
    \begin{equation*}
        0\xrightarrow[]{\subset} \B_{k+3}^{(4,2)}(F) \xrightarrow[]{\sym\grad_F\curl_F} \B_{k+1}^{(2,0)}(\rot_F\rot_F,F;\ms_2 \cap \mt_2)\xrightarrow[]{\rot_F\rot_F} \B_{k-1}^{(0)}(F) \cap P_{1}^{+,\perp}\xrightarrow[]{}0
    \end{equation*}
    is exact. Here, $\B_{k-1}^{(0)}(F) \cap P_{1}^{+,\perp}$ denotes the orthogonal complement space of $P_1^+$ with respect to $\B_{k-1}^{(0)}(F) $ under the $L^2$ inner product.
\end{lemma}
\begin{proof}
    Recall that in \Cref{uni:rotrot} we have shown that for $\bta\in P_{k+1}(F;\ms_2\cap\mt_2)$, there exists $ q\in\B_{k+3}^{(4,2)}(F)$ satisfying
    \begin{equation*}
        \sym\grad_F\curl_F q=\bta,
    \end{equation*}
    provided that the degrees of freedom \eqref{defcurl2:1}-\eqref{defcurl2:3} vanish for $\bta$ and $\rot_F\rot_F\bta=0$. This actually proves that 
    \begin{equation*}
        \B_{k+1}^{(2,0)}(\rot_F\rot_F,F;\ms_2 \cap \mt_2) \cap \ker(\rot_F\rot_F) = \sym\grad_F\curl_F \B_{k+3}^{(4,2)}(F).
    \end{equation*}
    Then, by dimension counting we can obtain
    \begin{align*}
        &\dim (\rot_F\rot_F \B_{k+1}^{(2,0)}(\rot_F\rot_F,F;\ms_2 \cap \mt_2))\\
        = &\dim \B_{k+1}^{(2,0)}(\rot_F\rot_F,F;\ms_2 \cap \mt_2) - \dim \B_{k+3}^{(4,2)}(F) \\
        = &\frac{1}{2}k(k+1)-7=\dim  \B_{k-1}^{(0)}(F)-4.
    \end{align*}
    Since the degrees of freedom \eqref{defcurl2:1}-\eqref{defcurl2:3} imply that $\rot_F\rot_F \B_{k+1}^{(2,0)}(\rot_F\rot_F,F;\ms_2 \cap \mt_2) \subseteq \B_{k-1}^{(0)}(F) \cap P_{1}^{+,\perp}$, the above identity concludes the proof.
\end{proof}

Given a bounded Lipschitz contractible polygonal domain $F \subset \mr^2$, the finite elements for the conformal strain complex \eqref{eq:complex2Db} can be constructed with the above shape function spaces and degrees of freedom. It can be verified that the discrete conformal strain complex is exact. The proof follows similar arguments as those in \cite[Lemma 3.10]{chen2020finite} and is omitted here. 

\subsection{Finite element de Rham complex in two dimensions}
This subsection constructs the finite elements with respect to the de Rham complex \eqref{eq:complex2Da} in two dimensions.

The shape function space of the $H^2$ conforming element is $P_{k+2}(F)$ with $k \geq 5$. For $p \in P_{k+2}(F)$, the degrees of freedom of the $H^2$ conforming element with enhanced continuity at vertices are defined as follows:
\begin{subequations}
\begin{align}
    D^{\alpha}p(x) &~~ \text{for all}\; x \in \mv(F), \, |\alpha| \leq 3; \label{derham1:1}\\
    (p,q)_e &~~ \text{for all}\; q \in P_{k-6}(e), \, e \in \me(F); \label{derham1:2}\\
    (\frac{\partial p}{\partial \mn}, q)_e &~~ \text{for all}\; q \in P_{k-5}(e), \, e\in\me(F); \label{derham1:3}\\
    (p,q)_F &~~ \text{for all}\; q\in P_{k-4}(F).\label{derham1:4}
\end{align}
\end{subequations}
The unisolvency of the degrees of freedom \eqref{derham1:1}-\eqref{derham1:4} can be proved by a standard argument and the details are omitted here. 

The shape function space of the $H^1(\rot_F;\mr^2)$ conforming element is $P_{k+1}(F)$ with $k\geq 4$, and the degrees of freedom read as follows:

\begin{subequations}
\begin{align}
    D^{\alpha}\bm{u}(x) &~~ \text{for all}\; x \in \mv(F), \, |\alpha| \leq 2; \label{derham2:1}\\
    (\bm{u},\bm{v})_e &~~ \text{for all}\; \bm{v} \in P_{k-5}(e;\mr^2), \, e \in \me(F); \label{derham2:2}\\
    (\rot_F \bm{u}, q)_e &~~ \text{for all}\; q \in P_{k-4}(e), \, e\in\me(F); \label{derham2:3}\\
    (\bm{u},\bm{v})_F &~~ \text{for all}\; \bm v\in \curl_F P_{k-3}(F)+\grad_F \B_{k+2}^{(3,1)}(F).\label{derham2:4}
\end{align}
\end{subequations}
An integration by parts shows that the two spaces in the degrees of freedom \eqref{derham2:4} are orthogonal, implying that  the associated degrees of freedom are linearly independent.
\begin{lemma}\label{uni:rot}
    The degrees of freedom \eqref{derham2:1}-\eqref{derham2:4} are unisolvent for $P_{k+1}(F;\mr^2)$.
\end{lemma}

\begin{proof}
    It can be verified that $
        \B_{k+ 
   2}^{(3,1)}(F)=(\lambda_1\lambda_2\lambda_3)^2 P_{k-4}(F). $ 
     Therefore, the number of degrees of freedom is equal to the dimension of $P_{k+1}(F;\mr^2)$, namely,
    \begin{equation*}
        36+6(k-4)+3(k-3)+\frac{(k-1)(k-2)}{2}-1+\frac{(k-2)(k-3)}{2}=(k+2)(k+3).
    \end{equation*}
    Suppose that \eqref{derham2:1}-\eqref{derham2:4} vanish for some $\bm u \in P_{k+1}(F;\mr^2)$. Then it holds that $\bm u \in \B_{k+1}^{(1,+)}(\rot_F,F;\mr^2)$ and $\rot_F \bm u$ vanishes on $\partial F$. This shows $
        \rot_F \bm u=\lambda_1\lambda_2\lambda_3 r$ 
    for some $r\in P_{k-3}(F)$. For any $q \in P_{k-3}(F)$, an integration by parts leads to
    \begin{equation*}
        \int_{F} \bm u \cdot \curl_F q=-\int_F \lambda_1\lambda_2\lambda_3 rq.
    \end{equation*}
    Since \eqref{derham2:4} vanishes for $\bm{u}$, it follows that $ r = 0$, thus $\rot_F\bm u=0$. 
    From the exactness of the bubble complex \eqref{derham:B}, we can derive that $
        \bm u=\grad_F p$ 
    for some $p \in \B_{k+2}^{(2,0)}(F)$. Since \eqref{derham2:1} and \eqref{derham2:2} imply that the second order derivatives of $\bm u$ vanishes at the vertices of $F$ and $\bm u$ vanishes on $\partial F$, the third order derivatives of $p$ vanish at the vertices and the first order derivatives of $p$ vanish on $\partial F$. Thus, $p\in \B_{k+1}^{(3,1)}(F)$, and the degrees of freedom \eqref{derham2:4} lead to $\bm u=\grad_F p=0$. This completes the proof.
\end{proof}

The shape function space of the scalar $H^1$ element is $P_k(F)$ with $k \geq 3$, and the degrees of freedom coincide with the Hermite element as follows:
\begin{subequations}
\begin{align}
    D^{\alpha}p(x) &~~ \text{for all}\; x \in \mv(F), \, |\alpha| \leq 1; \label{derham3:1}\\
    (p,q)_e &~~ \text{for all}\; q \in P_{k-4}(e), \, e \in \me(F); \label{derham3:2}\\
    (p,q)_F &~~ \text{for all}\; q\in P_{k-3}(F).\label{derham3:3}
\end{align}
\end{subequations}
 
Notice that the $H^2$ bubble function space of $P_{k+2}(F)$ with vanishing degrees of freedom \eqref{derham1:1}-\eqref{derham1:3} is exactly $\B_{k+2}^{(3,1)}(F)$, and the $H^1$ bubble function space of $P_{k}(F)$ with vanishing degrees of freedom \eqref{derham3:1}-\eqref{derham3:2} is $\B_{k}^{(1,0)}(F)$. Let $\B_{k+1}^{(2,0+)}(\rot_F,F;\mr^2)$ denote the $H(\rot_F,F;\mr^2)$ bubble function space of $P_{k+1}(F;\mr^2)$ with vanishing degrees of freedom \eqref{derham2:1}-\eqref{derham2:3}. These bubble functions form an exact complex as stated in the following lemma.
\begin{lemma}
    Suppose $k\geq 5$. The complex
    \begin{align*}
        0 \xrightarrow[]{\subset}  \B_{k+2}^{(3,1)}(F) \xrightarrow[]{\grad_F} \B_{k+1}^{(2,0+)}(\rot_F,F;\mr^2) \xrightarrow[]{\rot_F} \B_{k}^{(1,0)}(F) \cap P_0(F)^{\perp}\xrightarrow[]{}0
    \end{align*}
    is exact.
\end{lemma}
\begin{proof}
    Recall that in \Cref{uni:rot} we have shown that for $\bm u\in P_{k+1}(F;\mr^2)$, if \eqref{derham2:1}-\eqref{derham2:3} vanish for $\bm u$ and $\rot_F\bm u=0$, then there exists $p\in\B_{k+2}^{(3,1)}(F)$ satisfying $
        \grad_F p=\bm u,$ 
    This actually proves that 
    \begin{equation*}
    \B_{k+1}^{(2,0+)}(\rot_F,F;\mr^2) \cap \ker(\rot_F) = \grad_F \B_{k+2}^{(3,1)}(F)
    \end{equation*}
    Then, by dimension count we can obtain that
    \begin{align*}
        \dim (\rot_F \B_{k+1}^{(2,0+)}(\rot_F,F;\mr^2)) &= \dim \B_{k+1}^{(2,0+)}(\rot_F,F;\mr^2) - \dim  \B_{k+2}^{(3,1)}(F) \\
        &= \frac{1}{2}(k-1)(k-2)-1 = \dim \B_{k}^{(1,0)}(F)-1.
    \end{align*}
    Since the degrees of freedom \eqref{derham2:1}-\eqref{derham2:3} imply that $\rot_F \B_{k+1}^{(2,0+)}(\rot_F,F;\mr^2) \subseteq \B_{k}^{(1,0)}(F) \cap P_0(F)^{\perp}$, this identity concludes the proof.
\end{proof}

Given a bounded Lipschitz contractible polygonal domain $F \subset \mr^2$, the global finite element spaces shown above form an exact discrete complex, which can be proved by similar arguments as those in \cite{hu2023construction} and the details are omitted for brevity.

\section{Discrete conformal Hessian complex in three dimensions}\label{sec:FEM3D}

This section constructs the discrete conformal Hessian complex \eqref{dct_cplx} on simplicial grids in three dimensions:
\begin{align*}
    P_{1}^{+} \xrightarrow[]{\subset}  U_{k+3,h} \xrightarrow[]{\dev\hess} \Lambda_{k+1,h} \xrightarrow[]{\sym\curl} \Sigma_{k,h} \xrightarrow[]{\div\div} P_{k-2}(\mathcal{T}_h)\xrightarrow[]{} 0.
\end{align*}
We first define the degrees of freedom for the $H(\sym\curl;\mathbb{S}\cap \mathbb{T})$ conforming finite element space $\Lambda_{k+1,h}$ in \Cref{subsec:symcurl} and the $H(\div\div;\mathbb{S}\cap\mathbb{T})$ conforming finite element space $\Sigma_{k,h}$ in \Cref{subsec:divdiv}, based primarily on the bubble complexes in \Cref{sec:bubble} and the two-dimensional complexes in \Cref{sec:trace}. Then, we construct the finite element conformal Hessian complex \eqref{dct_cplx} and prove its exactness in \Cref{subsec:conformal_Hessian}.

\subsection{$H(\sym\curl;\ST)$ conforming finite element}\label{subsec:symcurl}
The construction of the $H(\sym\curl,\Omega;\\\ST)$ finite element space is based on the bubble function space $\B_{k+1}^{(2,0)}(\sym\curl,K;\ST)$ in \Cref{lem:symcurl_basis} from \Cref{sec:bubble} and the interior degrees of freedom of the finite elements in two dimensions from \Cref{sec:trace}. 

The degrees of freedom of the $H(\sym\curl,\Omega; \ST)$ finite element space are given as follows: for $\bm \tau \in P_{k+1}(K;\ST)$ with $k\ge 6$, 
\begin{subequations}
\begin{align}
    D^{\alpha}\bta(x)&~~ \text{for all}\; x\in \mv(K), \, |\alpha| \leq 2;\label{symcurl:1} \\
    (\bta, \bm{q})_{e}&~~ \text{for all}\; \bm{q} \in P_{k-5}(e; \ST),\, e\in \me(K);\label{symcurl:2} \\
    (\mn_i^T \sym\curl \bta \mn_j, q)_{e}&~~  \text{for all}\; q \in P_{k-4}(e), \, e\in \me(K),\, i,j=1,2;\label{symcurl:3} \\
    (\mn_1^T \curl \bta \mn_2-\partial_{\mbt}(\mbt^T \bta \mbt),q)_{e}&~~  \text{for all}\; q \in P_{k-4}(e), \, e\in \me(K);\label{symcurl:4} \\
    (\Pi_F(\bta^T \mn), \bm{q})_{F}&~~  \text{for all}\; \bm{q} \in \B_{k+1}^{(2,0)}(\rot_F, F; \mr^2), \, F\in \mf(K);\label{symcurl:5} \\
    (\sym(\Pi_F(\bta\times\mn)\Pi_F), \bm{q})_{F}&~~  \text{for all}\; \bm{q} \in \B_{k+1}^{(2,0)}(\rot_F\rot_F, F; \ms_2\cap\mt_2), \, F\in \mf(K);\label{symcurl:6} \\
    (\bta, \bm{q})_{K}&~~ \text{for all}\; \bm{q} \in \B_{k+1}^{(2,0)}(\sym\curl,K;\ST).\label{symcurl:7} 
\end{align}
\end{subequations}

\begin{theorem} 
\label{thm:uni_symcurl}
The degrees of freedom \eqref{symcurl:1}-\eqref{symcurl:7} are unisolvent for $P_{k+1}(K;\ST)$ with $k\ge 6$, and the resulting finite element space defined by
\begin{align*}
    \Lambda_{k+1,h}:=&\{\bta \in L^2(\Omega; \ST) \colon  \bta|_K \in P_{k+1}(K;\ST)\;\text{for all}\; K \in \mathcal{T}_h,\\ &\text{all of the degrees of freedom \eqref{symcurl:1}-\eqref{symcurl:6} are single-valued}\}.
\end{align*}
is $H(\sym \curl;\ST)$ conforming.
\end{theorem}

\begin{proof}
    Since the basis of the $H(\sym\curl;\ST)$ bubble function space has been constructed in \Cref{lem:symcurl_basis}, the dimension of $\B_{k+1}^{(2,0)}(\sym\curl;\ST)$ is 
    \begin{equation*}
        \dim \B_{k+1}^{(2,0)}(\sym\curl;\ST)=\frac{1}{6}(5k^3-3k^2-2k-72).
    \end{equation*}
    The number of degrees of freedom \eqref{symcurl:1}-\eqref{symcurl:7} is
    \begin{gather*}
        200+30(k-4)+18(k-3)+6(k-3)+2(k-2)(k-3)+2(k-1)(k-2)-4\\+2(k-4)(k-5)+2k(k+1)-28+ \frac{5k^3-3k^2-2k-72}{6}
        = \dim P_{k+1}(K;\mathbb{S}\cap\mathbb{T}).
    \end{gather*}
    
    It suffices to prove that for any $\bta\in P_{k+1}(K;\ms \cap \mt)$, if $\bta$ vanishes on all the degrees of  freedom of \eqref{symcurl:1}-\eqref{symcurl:7}, the $\bta=0$.
    
    The degrees of freedom \eqref{symcurl:1} and \eqref{symcurl:2} indicate that $\bta\in \B_{k+1}^{(2,0)}(K;\ms \cap \mt)$. 
    Given $F\in\mf(K)$, let $\bm{w}:=\Pi_F(\bta^T \mn)$ and $\bm{\zeta}:=\sym(\Pi_F(\bta\times\mn)\Pi_F)$. By \cite[Lemma 4.2]{Hu2022d}, we have
    \begin{align*}
        \rot_F \bm{w}&=\mn^T \sym\curl \bta \mn,\\
        -\partial_{\mbt_{\partial F}}(\mn_{\partial F}^T \bm{\zeta} \mbt_{\partial F})+\mbt_{\partial F}^T \rot_F \bm{\zeta}&=-\mn_{\partial F}^T \curl \bta \mn-\partial_{\mbt_{\partial F}}(\mbt_{\partial F}^T \bta \mbt_{\partial F}).
    \end{align*}
    These, together with the degrees of freedom \eqref{symcurl:3}, \eqref{symcurl:4}, indicate that $\rot_F \bm{w}=0$ and $-\partial_{\mbt_{\partial F}}(\mn_{\partial F}^T \bm{\zeta} \mbt_{\partial F})+\mbt_{\partial F}^T \rot_F \bm{\zeta}=0$ on $\partial F$ by similar arguments as those in \cite[Theorem4.5]{Hu2022d}. Hence, it follows that $\bm{w} \in \B_{k+1}^{(2,0)}(\rot_F, F; \mr^2)$ and $\bm{\zeta} \in \B_{k+1}^{(2,0)}(\rot_F\rot_F, F; \ms_2\cap\mt_2)$. A combination with \eqref{symcurl:5} and \eqref{symcurl:6} implies $\bm{w}=0$ and $\bm{\zeta}=0$ on $F$. Therefore $\bta \in \B_{k+1}^{(2,0)}(\sym\curl,K; \ST)$ and moreover equals zero by \eqref{symcurl:7}. This completes the proof.
\end{proof}

\subsection{$H(\div\div;\ST)$ conforming finite element}\label{subsec:divdiv}
The degrees of freedom of the \\$H(\div\div,\Omega; \ST)$ finite element space are stated as follows: for $\bm \sigma \in P_{k}(K;\ST)$ with $k\ge 6$,
\begin{subequations}
\begin{align}
    D^{\alpha}\bsi (x)&~~ \text{for all}\; x\in \mv(K), \, |\alpha| \leq 1; \label{divdiv:1}\\
    (\mn_i^T \bsi \mn_j, q)_e &~~ \text{for all}\; q \in P_{k-4}(e),\, e\in \me(K),\, i,j=1,2; \label{divdiv:2}\\
    (\mn^T\bsi\mn, q)_F &~~  \text{for all}\; q \in \B_{k}^{(1,0)}(F), \, F\in \mf(K);\label{divdiv:3} \\
    (2\div_F(\bsi\mn)+\partial_{\mn}(\mn^T\bsi\mn), q)_{F} &~~  \text{for all}\; q \in \B_{k-1}^{(0)}(F), \, F\in \mf(K); \label{divdiv:4}\\
    (\bsi, \dev \hess q)_{K}&~~  \text{for all}\; q \in P_{k-2}(K); \label{divdiv:5}\\
    (\bsi, \sym\curl\bta)_{K}&~~ \text{for all}\; \bta \in \B_{k+1}^{(2,0)}(\sym\curl,K;\ST). \label{divdiv:6}
\end{align}
\end{subequations}
\begin{theorem} 
\label{thm:uni_divdiv}
The degrees of freedom \eqref{divdiv:1}-\eqref{divdiv:6} are unisolvent for $P_{k}(K;\ST)$ with $k\ge 6$, and the corresponding finite element space defined by
\begin{align*}
    \Sigma_{k+1,h}:=&\{\bsi \in L^2(\Omega;\ST) \colon  \bsi|_K \in P_{k}(K;\ST)\;\text{for all}\; K \in \mathcal{T}_h,\\ &\text{all of the degrees of freedom \eqref{divdiv:1}-\eqref{divdiv:4} are single-valued}\}.
\end{align*}

is $H(\div\div;\ST)$ conforming.
\end{theorem}

\begin{proof}

    The number of all degrees of freedom \eqref{divdiv:1}-\eqref{divdiv:6} is
    \begin{gather*}
        80+18(k-3)+2(k-1)(k-2)+2k(k+1)-12\\
        +\frac{(k+1)k(k-1)}{6}-5+\frac{2k^3+3k^2-14k-24}{3}=\dim  \;P_k(K;\mathbb{S}\cap\mathbb{T}).
    \end{gather*}
    
    It suffices to show if the degrees of freedom \eqref{divdiv:1}-\eqref{divdiv:6} vanish for some $\bsi\in P_k(K;\mathbb{S}\cap\mathbb{T})$, then $\bsi=0$. The degrees of freedom \eqref{divdiv:1}-\eqref{divdiv:4} show that $\bsi$ is in $\B_k^{(1)}(\div\div, K;\ST)$. An integration by parts and \eqref{divdiv:5} imply that $\div\div \bsi=0$. Hence, by the exactness of \eqref{dct_divdivST}, there exists $\bta \in \B_{k+1}^{(2,0)}(\sym \curl,K;\ST)$ such that $\bsi = \sym\curl \bta$. Then it follows from the degrees of freedom \eqref{divdiv:6} that $\bsi = 0$. This completes the proof. 
\end{proof}

\begin{remark}\label{rmk:dof_divdiv}
    The degrees of freedom \eqref{divdiv:3} and \eqref{divdiv:4} can be rewritten as:
    \begin{align*}
        (\mn^T\bsi\mn, q)_F &~~  \text{for all}\; q \in P_{k-3}(F), \, F\in \mf(K);\tag{\ref*{divdiv:3}'} \\
        (2\div_F(\bsi\mn)+\partial_{\mn}(\mn^T\bsi\mn), q)_{F} &~~  \text{for all}\; q \in \B_{k-1}^{(0)}(F)\cap \B_5^{(1,0)}(F)^{\perp}+P_2(F), \, F\in \mf(K). \tag{\ref*{divdiv:4}'}
    \end{align*}
    Here we replace the bubble function spaces $\B_{k}^{(1,0)}(F)$ and $\B_{k-1}^{(0)}(F)$ in \eqref{divdiv:3} and \eqref{divdiv:4} by their corresponding dual spaces. The modified degrees of freedom define a flux preserving canonical interpolation, which will be used to prove the exactness of the conformal Hessian finite element complex in \Cref{subsec:conformal_Hessian}. It is easy to prove that \Cref{thm:uni_divdiv} also holds for the modified degrees of freedom.
\end{remark}

\subsection{Finite element conformal Hessian complex}\label{subsec:conformal_Hessian}
Before establishing the finite element conformal Hessian complex \eqref{dct_cplx}, it remains to construct the rest two scalar-valued finite element spaces. 
Recall the $H^2$ conforming finite element space from~\cite{hu2023construction,Zhang2009,MR275014}. The shape function space is $P_{k+3}(K)$ with $k\geq 6$, and the degrees of freedom are defined by the following:
\begin{subequations}
\begin{align}
    D^{\alpha}u (x)&~~ \text{for all}\; x\in \mv(K), \, |\alpha| \leq 4; \label{devhess:1}\\
    (u, q)_e &~~ \text{for all}\; q \in P_{k-7}(e),\, e\in \me(K); \label{devhess:2}\\
    (\frac{\partial u}{\partial \mn_{i}}, q)_e &~~  \text{for all}\; q \in P_{k-6}(e), \, e\in \me(K),\, i=1,2;\label{devhess:3} \\
    (\frac{\partial^2 u}{\partial \mn_{i} \partial \mn_{j}}, q)_{e} &~~  \text{for all}\; q \in P_{k-5}(e), \, e\in \me(K),\, i,j=1,2; \label{devhess:4}\\
    (u, q)_{F}&~~  \text{for all}\; q \in P_{k-6}(F),\, F\in\mf(K); \label{devhess:5}\\
    (\frac{\partial u}{\partial \mn},q)_{F}&~~ \text{for all}\; q \in P_{k-4}(F),\, F\in\mf(K); \label{devhess:6}\\
    (u,q)_K &~~ \text{for all}\; q \in P_{k-5}(K). \label{devhess:7}
\end{align}
\end{subequations}
The $H^2$ conforming finite element space $U_{k+3,h}$ is then defined by
\begin{align*}
    U_{k+3,h}:=&\{u \in H^2(\Omega) \colon  u|_K \in P_{k+3}(K)\;\text{for all}\; K \in \mathcal{T}_h,\\ &\text{all of the degrees of freedom \eqref{devhess:1}-\eqref{devhess:6} are single-valued}\}.
\end{align*}

The last space $P_{k-2}(\mathcal{T}_h)$ in \eqref{dct_cplx} consists of discontinuous piecewise polynomials of degree no greater than $k-2$.

The following exactness result holds for the finite element conformal Hessian complex \eqref{dct_cplx}.
\begin{theorem}[Exactness of finite element complex] \label{thm:conformal_exact}
    Suppose that $k\geq 6$. The following sequence forms an exact complex:
    \begin{align*}
        P_{1}^{+} \xrightarrow[]{\subset}  U_{k+3,h} \xrightarrow[]{\dev\hess} \Lambda_{k+1,h} \xrightarrow[]{\sym\curl} \Sigma_{k,h} \xrightarrow[]{\div\div} P_{k-2}(\mathcal{T}_h)\xrightarrow[]{} 0,
    \end{align*}
    provided that $\mathcal T_h$ is a triangulation of a bounded Lipschitz contractible domain $\Omega$.
\end{theorem}
\begin{proof}
    Suppose $u\in U_{k+3,h}$ and $\bta=\dev\hess u$. It follows from \cite[Lemma 4.10]{Hu2022d} that on any $e\in\me$ the following identities hold:
    \begin{align*}
        \mn_i^T\sym\curl\bta\mn_j = 0,\;i,j=1,2,\\
        \mn_1^T\curl\bta\mn_2-\partial_{\mbt}(\mbt^T\bta\mbt) = -\partial_{\mbt}^3 u.
    \end{align*}
    This shows that the edge degrees of freedom \eqref{symcurl:3} and \eqref{symcurl:4} are single-valued for $\dev\hess u$ and hence $\dev\hess U_{k+3,h} \subset \Lambda_{k+1,h}$. It can be proved that the inclusion $\sym\curl \Lambda_{k+1,h} \subset \Sigma_{k,h}$ holds as well. Therefore, the sequence \eqref{dct_cplx} forms a complex. 
    
    It remains to prove the exactness, which is equivalent to the following statements:
    \begin{enumerate}
        \item $ U_{k+3,h} \cap \ker(\dev\hess)=P_{1}^{+}$,
        \item $\Lambda_{k+1,h} \cap\ker(\sym\curl) = \dev\hess U_{k+3,h}$,
        \item $P_{k-2}(\mathcal{T}_h) =\div\div\Sigma_{k,h}$,
        \item $\Sigma_{k,h}\cap \ker(\div\div) = \Lambda_{k+1,h}$.
    \end{enumerate}

   First, we show that 
   $ U_{k+3,h} \cap \ker(\dev\hess)=P_{1}^{+}$. By the exactness of the conformal Hessian complex \eqref{eq:complex3Dc},
    \begin{align}
        P_{1}^{+} \subset U_{k+3,h} \cap \ker(\dev\hess) \subset H^2(\Omega) \cap \ker(\dev\hess) =P_{1}^{+}.
    \end{align}

    Second, we show that $\Lambda_{k+1,h} \cap\ker(\sym\curl) = \dev\hess U_{k+3,h}$, namely, if $\sym\curl\bta=0$ and $\bta\in \Lambda_{k+1,h}$, then there exists $u\in U_{k+3,h}$, such that $\bta=\dev\hess u$.

    Since $\sym\curl\bta=0$, by the exactness of \eqref{eq:complex3Dc}, there exists $u\in H^2(\Omega)$ such that $\bta=\dev\hess u$. Fix any $K\in \mathcal{T}_h$, $\div\bta=\frac{2}{3}\grad(\Delta u)$ is a polynomial. Hence, $\Delta u$ is a polynomial and $\hess u=\dev\hess u+\frac{1}{3}\Delta u\mI$ is a polynomial. This leads to $u|_K\in P_{k+3}(K)$. It follows from  $u$ is piecewise smooth and globally $H^2$ that $u$ is $C^1$ at vertices, and on the edges and faces.
    
    Now it suffices to show that $u \in U_{k+3,h}$ by showing that $u$ has the corresponding continuity. Since $\bta$ is $C^2$ at the vertices, it then holds that $u$ is $C^4$ at the vertices. For $e\in\me$ with unit tangential vector $\mbt$, $\bta$ is continuous across $e$. Therefore, $\mbt^T\bta\mbt=\partial_{\mbt}^2 u-\frac{1}{3}\Delta u$ is continuous across $e$, which shows that $\Delta u$ is $C^0$ on the edges. Then $\hess u$ is $C^0$ on the edges and $u$ is $C^2$ on the edges. These imply $u\in U_{k+3,h}$ and consequently $\Lambda_{k+1,h} \cap\ker(\sym\curl) = \dev\hess U_{k+3,h}$.

    Third, we show that $P_{k-2}(\mathcal{T}_h) = \div\div\Sigma_{k,h}$, namely, if $q \in P_{k-2}(\mathcal{T}_h)$, then there exists $\bsi \in \Sigma_{k,h}$ s.t. $q=\div\div\bsi$.

    It follows from \cite[(52)]{Arnold2021} that $\div\div H^2(\Omega;\ST)=L^2(\Omega)$. That is, there exists $\bta \in H^2(\Omega;\ST)$ such that $
        \div\div\bta=q.
$ 
Let $I_h \bta \in\Sigma_{k,h}$ be determined by
$
        N(I_h \bta)=N(\bta)
$ 
    for all the degrees of  freedom $N$ from \eqref{divdiv:1} to \eqref{divdiv:6}, except that the nodal values in \eqref{divdiv:1} are determined by the Scott-Zhang interpolation~\cite{Scott1990}, and~\eqref{divdiv:3} and \eqref{divdiv:4} are modified as in \Cref{rmk:dof_divdiv}. Note that the degrees of freedom on the edges and faces are well-defined for $H^2$ functions.
    An integration by parts shows that
    \begin{align*}
        (q-\div\div I_h\bta, p)_K = (\bta-I_h\bta, \dev\hess p)_K = 0,\; \forall p\in P_{1}^{+},\, K\in\mathcal{T}_h.
    \end{align*}
    Hence $(q-\div\div I_h\bta)|_K\in P_{k-2}(K)\cap P_{1}^{+,\perp}$. By the exactness of the conformal Hessian bubble complex \eqref{dct_divdivST}, there exists $\bta_b\in \Sigma_{k,h}$ such that $\bta_b|_K\in \B_k^{(1)}(\div\div,K;\ST)$ for each $K\in \mathcal{T}_h$, and 
    \begin{equation*}
        q-\div\div I_h\bta=\div\div\bta_b.
    \end{equation*}
    Therefore $q=\div\div (I_h\bta+\bta_b)$, where $I_h\bta+\bta_b\in\Sigma_{k,h}$ as required.

   Finally, we show that $\Sigma_{k,h}\cap \ker(\div\div) = \Lambda_{k+1,h}$.

    We prove this identity by dimension count. By the degrees of freedom of the finite element spaces, we obtain 
    \begin{align*}
        \dim U_{k+3,h} = 35\#\mv + (6k-28)\#\me + (k^2-7k+13)\#\mf + \frac{k^3-9k^2+26k-24}{6}\#\mathcal{T}_h,\\
        \dim \Lambda_{k+1,h} = 50\#\mv + (9k-32)\#\me +(2k^2-8k+6)\#\mf + \frac{5k^3-3k^2-2k-72}{6}\#\mathcal{T}_h,\\
        \dim \Sigma_{k,h} = 20\#\mv + (3k-9)\#\me +(k^2-k-2)\#\mf + \frac{5k^3+6k^2-29k-78}{6}\#\mathcal{T}_h.
    \end{align*}
    By Euler's formula $\#\mv -\#\me +\#\mf -\#\mathcal{T}_h=1$, it holds that
    \begin{align*}
        \dim U_{k+3,h}-\dim \Lambda_{k+1,h}+\dim \Sigma_{k,h}-\dim  P_{k-2}(\mathcal{T}_h) = 5 =\dim P_{1}^{+}.
    \end{align*}
    This concludes that the discrete complex \eqref{dct_cplx} is exact.
\end{proof}

\section{Error estimates of the linearized Einstein-Bianchi system}\label{sec:application}
In this section, we show that the finite element complex \eqref{dct_divdivST} constructed in the previous sections can be used for the discretization of the linearized Einstein-Bianchi system \eqref{eq:EB}. We deal with the linear system \eqref{eq:EB} as the Hodge wave equation \eqref{eq:HodgeEB2} by smilar methods as in \cite{Hu2021a,Hu2022d}. 

Recall that the linearized Einstein-Bianchi system \eqref{eq:EB} can be realized as a Hodge wave equation \eqref{eq:HodgeEB2}
\begin{equation*}
\begin{aligned}
    \begin{cases}
        \dot{\sigma}=\div\div\bm{E},\\
        \dot{\bm{E}}=-\dev\hess\sigma-\sym\curl\bm{B},\\
        \dot{\bm{B}}=\sym\curl\bm{E}.
    \end{cases}
\end{aligned}
\end{equation*}
The weak formulation of the Hodge wave equation is proposed as follows: Find $\bm \sigma$, $\bm E$, $\bm B$ such that
\begin{equation}
    \begin{aligned}
        \sigma&\in C^1([0,T],L^2(\Omega;\mr)),\\
        \bm{E}&\in C^0([0,T],H(\div\div,\Omega;\ST)) \cap C^1([0,T],L^2(\Omega;\ST)),\\
        \bm{B}&\in C^0([0,T],H(\sym\curl,\Omega;\ST)) \cap C^1([0,T],L^2(\Omega;\ST)),
    \end{aligned}
\end{equation}
satisfying
\begin{equation}\label{eq:wkform1}
    \begin{aligned}
        \begin{cases}
            (\dot{\sigma},\tau)=(\div\div\bm{E},\tau), &\forall \tau\in L^2(\Omega;\mr),\\
            (\dot{\bm{E}},\bm{\xi})=-(\sigma,\div\div\bm{\xi})-(\sym\curl\bm{B},\bm{\xi}), &\forall \bm{\xi}\in H(\div\div,\Omega;\ST),\\
            (\dot{\bm{B}},\bm{\zeta})=(\bm{E},\sym\curl\bm{\zeta}), &\forall \bm{\zeta} \in H(\sym\curl,\Omega;\ST),
        \end{cases}
    \end{aligned}
\end{equation} 
with given initial data $(\sigma(0),\bm{E}(0),\bm{B}(0))\in L^2(\Omega;\mr)\times H(\div\div,\Omega;\ST)\times H(\sym\curl,\Omega;\ST)$.

For $k\geq 6$, the semidiscretization of \eqref{eq:wkform1} finds
\begin{equation*}
    \sigma_h\in C^1([0,T],P_{k-2}(\mathcal{T}_h)),\quad \bm E_h\in C^0([0,T],\Sigma_{k,h}),\quad\text{and}\quad\bm B_h\in C^0([0,T],\Lambda_{k+1,h}),
\end{equation*}
such that 
\begin{equation}\label{eq:semidctform1}
    \begin{aligned}
        \begin{cases}
            (\dot{\sigma_h},\tau)=(\div\div\bm{E}_h,\tau), &\forall \tau\in P_{k-2}(\mathcal{T}_h),\\
            (\dot{\bm{E}_h},\bm{\xi})=-(\sigma_h,\div\div\bm{\xi})-(\sym\curl\bm{B}_h,\bm{\xi}), &\forall \bm{\xi}\in \Sigma_{k,h},\\
            (\dot{\bm{B}_h},\bm{\zeta})=(\bm{E}_h,\sym\curl\bm{\zeta}), &\forall \bm{\zeta} \in\Lambda_{k+1,h},
        \end{cases}
    \end{aligned}
\end{equation}
for all $t\in(0,T]$ with given initial data. 

For simplicity of notation, let $\bm V:=L^2(\Omega;\mr)\times H(\div\div,\Omega;\ST)\times H(\sym\curl,\Omega;\ST)$ with the norm 
\begin{equation*}
    \|(q,\bm\xi,\bm\zeta)\|_{\bm V}:=\|q\|_{L^2(\Omega)}+\|\bm\xi\|_{H(\div\div,\Omega)}+\|\bm\zeta\|_{H(\sym\curl,\Omega)}
\end{equation*}
for $(q,\bm\xi,\bm\zeta)\in\bm V$. Let $\bm V_h:=P_{k-2}(\mathcal{T}_h)\times\Sigma_{k,h}\times\Lambda_{k+1,h}$ be the corresponding  discrete space. Based on the semidiscretization scheme \eqref{eq:semidctform1}, we shall use the usual Crank-Nicolson scheme to discretize the time variable. Suppose that $T=N\Delta t$ with a positive integer $N$. Let $p^j$ denote the function $p(t_j)$ with $t_j=j\Delta t$ for $j=0,1,\dots,N$. Define
\begin{equation*}
    \partial_t p^{j+\frac{1}{2}}=\frac{p^{j+1}-p^j}{\Delta t},\;\hat{p}^{j+\frac{1}{2}}=\frac{p^{j+1}+p^j}{2}.
\end{equation*}
Denote by $(\sigma_h^j,\bm E_h^j,\bm B_h^j)\in\bm V_h$ the approximation of solution $(\sigma,\bm E,\bm B)$ of \eqref{eq:wkform1} at $t_j$. Given the initial data $(\sigma_h^0,\bm E_h^0,\bm B_h^0)\in\bm V_h$, for $0\leq j\leq N-1$, the approximation $(\sigma_h^{j+1},\bm E_h^{j+1},\bm B_h^{j+1})$ at $t_{j+1}$ is defined by the following Crank-Nicolson scheme:
\begin{equation}\label{eq:dctform1}
    \begin{aligned}
        \begin{cases}
            (\partial_t\sigma_h^{j+\frac{1}{2}},\tau)=(\div\div\hat{\bm{E}}_h^{j+\frac{1}{2}},\tau), &\forall \tau\in P_{k-2}(\mathcal{T}_h),\\
            (\partial_t\bm{E}_h^{j+\frac{1}{2}},\bm{\xi})=-(\hat{\sigma}_h^{j+\frac{1}{2}},\div\div\bm{\xi})-(\sym\curl\hat{\bm{B}}_h^{j+\frac{1}{2}},\bm{\xi}), &\forall \bm{\xi}\in \Sigma_{k,h},\\
            (\partial_t\bm{B}_h^{j+\frac{1}{2}},\bm{\zeta})=(\hat{\bm{E}}_h^{j+\frac{1}{2}},\sym\curl\bm{\zeta}), &\forall \bm{\zeta} \in\Lambda_{k+1,h}.
        \end{cases}
    \end{aligned}
\end{equation}
This can be written as
\begin{align*}
\begin{cases}
(\sigma_h^{j+1},q)-\frac{\Delta t}{2}(\div\div\bm E_h^{j+1},q)\\\quad\quad=(\sigma_h^{j},q)+\frac{\Delta t}{2}(\div\div\bm E_h^{j},q),&\forall q\in P_{k-2}(\mathcal{T}_h),\\
(\bm E_h^{j+1} ,\bm\xi)+\frac{\Delta t}{2}(\sigma_h^{j+1} ,\div\div\bm\xi)+\frac{\Delta t}{2}(\sym\curl \bm B_h^{j+1},\bm\xi)\\
\quad\quad=(\bm E_h^{j } ,\bm\xi)-\frac{\Delta t}{2}(\sigma_h^{j} ,\div\div\bm\xi)-\frac{\Delta t}{2}(\sym\curl \bm B_h^{j},\bm\xi),&\forall \bm\xi\in \Sigma_{k,h},\\
( \bm B_h^{j+1} ,\bm\zeta)-\frac{\Delta t}{2}( \bm E_h^{j+1},\sym\curl\bm\zeta)\\\quad\quad=( \bm B_h^{j } ,\bm\zeta)+\frac{\Delta t}{2}( \bm E_h^{j },\sym\curl\bm\zeta),&\forall\bm\zeta\in \Lambda_{k+1,h}.
\end{cases}
\end{align*}
The system is nonsingular as in \cite{Hu2021a}. 

To determine the initial data $(\sigma_h^0,\bm E_h^0,\bm B_h^0)$, we further introduce the following bilinear form $\mathcal{A}$ and the projection operator $\Pi_h$. Define the bilinear form $\mathcal{A}:\bm V\times\bm V\xrightarrow[]{} \mr$ by
\begin{align*}
    \mathcal{A}(\sigma,\bm E,\bm B;q,\bm\xi,\bm\zeta):=&(\sigma,q)+(\bm E,\bm\xi)+(\bm B,\bm\zeta)-(\div\div\bm E,q)+(\sigma,\div\div\bm\xi)\\&+(\sym\curl\bm B,\bm\xi)-(\bm E,\sym\curl\bm\zeta).
\end{align*}
By a similar argument as in \cite[Theorem 5.2]{Hu2022d}, we can prove that the inf-sup condition of $\mathcal{A}$ in $\bm V_h\times\bm V_h$ holds true:
\begin{equation*}
    \inf_{0\neq(\sigma,\bm E,\bm B)\in\bm V_h}\sup_{0\neq(q,\bm\xi,\bm\zeta)\in\bm V_h}\frac{\mathcal{A}(\sigma,\bm E,\bm B;q,\bm\xi,\bm\zeta)}{\|(\sigma,\bm E,\bm B)\|_{\bm V}\|(q,\bm\xi,\bm\zeta)\|_{\bm V}}=\beta>0
\end{equation*}
with constant $\beta$ independent of $h$. Hence, for any $(\sigma,\bm E,\bm B)\in \bm V$, we can define the elliptic projection $\Pi_h(\sigma,\bm E,\bm B)\in\bm V_h$ such that
\begin{equation}
    \mathcal{A}(\Pi_h\sigma,\Pi_h\bm E,\Pi_h\bm B;q,\bm\xi,\bm\zeta)=\mathcal{A}(\sigma,\bm E,\bm B;q,\bm\xi,\bm\zeta)\; \text{for any}\; (q,\bm\xi,\bm\zeta)\in\bm V_h.
\end{equation}
Then the discrete initial value $(\sigma_h^0,\bm E_h^0,\bm B_h^0)$ can be taken as $\Pi_h(\sigma(0),\bm E(0),\bm B(0))$, where $(\sigma,\bm E,\bm B)$ are the exact solutions of \eqref{eq:wkform1}. The error estimates are stated in the following proposition and the proof, following the same argument as in \cite[Theorem 5.3]{Hu2022d}, is omitted here.

\begin{prop}
    Suppose $k\geq 6$.  Let $(\sigma,\bm E,\bm B)$ solve \eqref{eq:wkform1}, let $(\sigma_h^j,\bm E_h^j,\bm B_h^j)$ solve \eqref{eq:dctform1}, and let the initial data $(\sigma_h^0,\bm E_h^0,\bm B_h^0)=\Pi_h(\sigma(0),\bm E(0),\bm B(0))$. Assume
    \begin{gather*}
        \sigma\in W^{1,1}\big([0,T], H^{k-1}(\Omega) \big)\cap W^{3,1}\big([0,T],L^2(\Omega)\big)\cap L^\infty\big([0,T],H^{k-1}(\Omega)\big),\\
        \bm E\in W^{1,1}\big([0,T], H^{k+1}(\Omega;\ST)\big)\cap W^{3,1}\big([0,T],L^2(\Omega;\ST)\big)\cap L^\infty\big([0,T],H^{k+1}(\Omega;\ST)\big),\\
        \bm B\in W^{1,1}\big([0,T], H^{k}(\Omega;\ST) \big)\cap W^{3,1}\big([0,T],L^2(\Omega;\ST)\big)\cap L^\infty\big([0,T],H^{k}(\Omega;\ST)\big).        
    \end{gather*}  
    For $1\leq j\leq N$,
    \begin{align*}
        &\|\sigma^j-\sigma_h^j\|_{L^2(\Omega)}+\|\bm E^j-\bm E_h^j\|_{L^2(\Omega)}+\|\bm B^j-\bm B_h^j\|_{L^2(\Omega)}\\
        \lesssim&(h^{k-1}+\Delta t^2)\big(\|{\sigma}\|_{W^{1,1}(H^{k-1})\cap W^{3,1}(L^2)\cap L^\infty(H^{k-1})}\\&+\|{\bm E}\|_{W^{1,1}(H^{k+1})\cap W^{3,1}(L^2)\cap L^\infty(H^{k+1})}\\
        &+\|{\bm B}\|_{W^{1,1}(H^{k })\cap W^{3,1}(L^2)\cap L^\infty(H^{k })}\big).
    \end{align*}
\end{prop}

\begin{remark}
We can also consider the dual formulation: find
    \begin{equation}
        \begin{aligned}
            \sigma&\in C^0([0,T],H^2(\Omega;\mr)) \cap C^1([0,T],L^2(\Omega;\mr)),\\
            \bm{E}&\in C^0([0,T],H(\sym\curl,\Omega;\ST)) \cap C^1([0,T],L^2(\Omega;\ST)),\\
            \bm{B}&\in C^1([0,T],L^2(\Omega;\ST)),
        \end{aligned}
    \end{equation}
    such that
    \begin{equation}\label{eq:wkform2}
        \begin{aligned}
            \begin{cases}
                (\dot{\sigma},\tau)=(\bm{E},\dev\hess\tau), &\forall \tau\in H^2(\Omega;\mr),\\
                (\dot{\bm{E}},\bm{\xi})=-(\dev\hess\sigma,\bm{\xi})-(\bm{B},\sym\curl\bm{\xi}), &\forall \bm{\xi}\in H(\sym\curl,\Omega;\ST),\\
                (\dot{\bm{B}},\bm{\zeta})=(\sym\curl\bm{E},\bm{\zeta}), &\forall \bm{\zeta} \in L^2(\Omega;\ST),
            \end{cases}
        \end{aligned}
    \end{equation}
with given initial data $(\sigma(0),\bm{E}(0),\bm{B}(0))\in H^2(\Omega;\mr)\times H(\sym\curl,\Omega;\ST)\times L^2(\Omega;\ST)$. The weak formulations \eqref{eq:wkform1} and \eqref{eq:wkform2} are both related to the conformal Hessian complex \eqref{eq:complex3Dc}. Here we only state the discretization scheme and error estimates for \eqref{eq:wkform1}, and the dual formulation \eqref{eq:wkform2} can be solved by similar methods. 
\end{remark}

\nocite{Guo2025a}

\bibliographystyle{plain}
\bibliography{reference}
\end{document}